\documentclass[11pt,reqno,oneside]{amsart}
\usepackage{comment}
\usepackage{enumerate}
\usepackage{geometry}
\usepackage{amsmath}
\usepackage{amsthm}
\usepackage{thmtools}
\usepackage{etoolbox}
\usepackage{hyperref}
\usepackage[capitalize]{cleveref}
\usepackage{amssymb}
\usepackage{mathrsfs}
\usepackage{lipsum}
\usepackage{graphicx}
\usepackage{subcaption}
\graphicspath{{./figure/}}
\DeclareGraphicsExtensions{.pdf,.jpeg,.png,.jpg}
\usepackage{xcolor}
\usepackage{tikz}
\usepackage{tkz-euclide}
\usepackage{changepage}

\usepackage{titletoc}
\usepackage{float}
\newtheoremstyle{noparen}
{3pt}{3pt}        
{\itshape}        
{}                
{\bfseries}       
{}               
{ }               
{\thmname{#1}\thmnumber{ #2} \thmnote{\normalfont #3}}

\theoremstyle{noparen}

\numberwithin{equation}{section}

\newtheorem{theorem}{Theorem}[section]
\newtheorem{lemma}[theorem]{Lemma}
\newtheorem{proposition}[theorem]{Proposition}
\newtheorem{corollary}[theorem]{Corollary}

\newtheoremstyle{nopunct}
{3pt}{3pt}        
{}                
{}                
{\bfseries}       
{}                
{ }               
{\thmname{#1}\thmnumber{ #2} \thmnote{ #3}}
\theoremstyle{nopunct}

\newtheorem{remark}[theorem]{Remark}

\newtheorem{definition}[theorem]{Definition}

\newtheorem{example}[theorem]{Example}

\begin{document}
	\title{Leibenson's equation on graphs}
\author{Philipp S\"{u}rig}
   \address{Philipp S\"{u}rig, Universit\"{a}t Bielefeld, Fakult\"{a}t f\"{u}r Mathematik, Postfach 100131, D-33501, Bielefeld, Germany}
    \email{philipp.suerig@uni-bielefeld.de}
\author{Xinrong Zhao}
    \address{Xinrong Zhao,
School of Mathematical Sciences, Fudan University, Shanghai, 200433, P.R. China}
    \email{xrzhao24@m.fudan.edu.cn}    
    \date{August 2026}

\thanks{The first author was funded by the Deutsche Forschungsgemeinschaft (DFG,
German Research Foundation) - Project-ID 317210226 - SFB 1283.}

	\begin{abstract}
		In this paper we study on infinite graphs the Leibenson equation
        $$
        \partial_t u = \Delta_p u^q,
        $$ 
        where $p>1$, $q>0$ and $\Delta_p$ denotes the discrete $p$-Laplacian. We prove, for any integrable initial data $u_0$, the existence of a global solution, which is unique for a certain range of $p$ and $q$. Assuming a \textit{Faber--Krahn inequality}, we obtain sharp $\ell^1$-$\ell^\infty$ smoothing estimates and quantitative bounds on the propagation of solutions with initially finite support. Under certain assumptions on $p$ and $q$, we also prove finite-time extinction results for solutions when the graph satisfies an \textit{isoperimetric inequality}. In particular, on Cayley graphs with polynomial volume growth, we establish the optimal large-time decay rate of the $\ell^\infty$-norm for nonnegative finite-mass solutions when $q(p-1)>1$, and demonstrate a sharp dichotomy regarding the finite-time extinction of exhaustion solutions.
	\end{abstract}

    \maketitle
	
	\tableofcontents	

	\section{Introduction}\label{sec:introduction}
    The nonlinear evolution equation
    \begin{equation}\label{eq:LeibensonEquation}
      \partial_t u = \Delta_p u^q,
    \end{equation}
    where $p>1$, $q>0$, and $\Delta_p u:=\operatorname{div}(|\nabla u|^{p-2}\nabla u)$, has attracted considerable interest in recent years. The continuous theory of nonlinear diffusion equations of porous-medium, fast-diffusion, and doubly nonlinear type is by now well developed; see, for example, \cite{DiBenedetto,Juan1,Juan2,Duzaar,Misawa} and the references therein.
    On Riemannian manifolds, this equation has recently been studied in several directions, including existence theory, finite propagation and upper bound estimates; see e.g., \cite{grigor2024finite,grigor2024sharp,grigor2025upper,Suerig1}. 
    
    Recently, there has been growing interest in the study of various differential equations on graphs. When $q=1$, the equation \eqref{eq:LeibensonEquation} reduces to the parabolic $p$-Laplace equation. A general semigroup approach to parabolic $p$-Laplace equation on graphs was developed by Mugnolo \cite{Mugnolo}. Further qualitative properties of solutions on infinite graphs were obtained by Hua and Mugnolo \cite{Hua1}. In particular, they proved that the solution to the parabolic $p$-Laplace equation vanishes in finite time for small $p$ under suitable isoperimetric assumptions and the solutions always satify conservation of mass for large $p$. In \cite{Dan}, Andreucci and Tedeev studied the parabolic $p$-Laplace equation on infinite graphs satisfying Faber--Krahn type inequalities. They proved sharp supremum bounds for the large time behavior of finite mass solutions and obtained the matching lower bounds via an optimal estimate for the effective speed of propagation of mass. For properties of solutions of (\ref{eq:infinite-parabolic}) in the case $q=1$ and $p=2$, that is, the solutions of \textit{discrete heat equation}, we refer to \cite{grigorgraph, Sun}. In this case, the relation between the Faber--Krahn inequality and solutions of \eqref{eq:infinite-parabolic} have been investigated in \cite{barlow2001manifolds, coulhon1998random}.

    The purpose of this paper is to extend the above results from the parabolic $p$-Laplace equation to the general Leibenson equation with arbitrary $p>1$ and $q>0$. 
    Assume that $G=(V,E,\omega)$ is a connected, countable, weighted graph with no multiple edges or self-loops. Here $V$ is the vertex set, $E$ is the edge set, and the weight function $\omega:V\times V\to \mathbb{R}$ is symmetric, nonnegative, and satisfies $\omega(x,y)>0$ if and only if $\{x,y\}\in E$. We write $y\sim x$ when $\{x,y\}\in E$ and abbreviate $\omega(x,y)$ as $\omega_{xy}$.  
    Define
    $$
        \mu(x) = \sum_{y\sim x} \omega_{xy}.
    $$
    We further assume that every vertex has finite degree and that $\mu_0:=\inf_{x\in V} \mu(x) > 0$.
    
    For an infinite graph $G$, we consider the discrete parabolic equation
    \begin{equation}\label{eq:infinite-parabolic}
        \begin{cases}
            \dfrac{\partial u}{\partial t}(x,t) = \Delta_p u^q(x,t), & x\in V,\ t>0,\\[8pt]
            u(x,0) = u_0(x), & x\in V,
        \end{cases}
    \end{equation}
    where
    $$
        \Delta_p u^q(x,t) = \frac{1}{\mu(x)} \sum_{y \in V} \bigl| u^q(y,t) - u^q(x,t) \bigr|^{p-2} \bigl( u^q(y,t) - u^q(x,t) \bigr) \, \omega_{xy}.
    $$
    Throughout we assume $p>1$, $q>0$, and for $k\in\mathbb{R}$, $s>0$ we adopt the convention $k^s := |k|^{s-1}k$. Moreover, to simplify the notation, for a function $f:V\to\mathbb{R}$ we set
    $$
        \nabla_{xy} f = f(y) - f(x),
    $$
    so that
    $$
        \Delta_p u^q(x,t) = \frac{1}{\mu(x)} \sum_{y \in V} (\nabla_{xy} u^q)^{p-1}  \, \omega_{xy}.
    $$
    
    \begin{definition}\label{def:infinite-solution}
        We say that a function $u$ is a \emph{solution} to \eqref{eq:infinite-parabolic} if there exist $T>0$ and $r\ge 1$ such that $u \in L^{\infty}\bigl([0,T); \ell^r(V)\bigr)$, $u(x,\cdot) \in C^{1}([0,T))$ for all $x\in V$, and $u$ satisfies \eqref{eq:infinite-parabolic} pointwise.
    \end{definition}
    
    We now introduce some further notation.
    Let $d(x,y)$ be the graph distance between $x$ and $y$, that is, the minimum number of edges in a path joining $x$ and $y$ in $G$. For $R\ge 0$ set
    $$
        B_R(x_0) = \{ x\in V : d(x,x_0) \le R \}.
    $$
    When the center $x_0$ is clear from the context, we simply write $B_R$.
    
    For a function $f:V\to\mathbb{R}$, $r\ge 1$, and $U\subset V$, we define the weighted $\ell^r$‑norm
    $$
        \|f\|_{\ell^r(U)} = \Bigl( \sum_{x\in U} |f(x)|^r \mu(x) \Bigr)^{\!1/r}, \quad
        \|f\|_{\ell^\infty(U)} = \sup_{x\in U} |f(x)|.
    $$
    The vertex and edge measures of $G$ are defined as
    $$
        \mu_V(U) = \sum_{x\in U} \mu(x) \quad (U\subset V), \quad
        \mu_E(K) = \sum_{\{x,y\}\in K} \omega_{xy} \quad (K\subset E).
    $$

    \begin{definition}\label{def:exhaustion-solution}
    Fix $x_0\in V$. For each $n\ge 1$, consider the following mixed problem:
    \begin{equation}\label{eq:finite-ball-problem}
        \begin{cases}
        \displaystyle
        \frac{\partial u_n}{\partial t}(x,t)
        =
        \Delta_p u_n^q(x,t),
        &
        x\in B_n(x_0),\ t>0,
        \\[1ex]
        u_n(x,0)=u_0(x),
        &
        x\in B_n(x_0),
        \\[1ex]
        u_n(x,t)=0,
        &
        x\notin B_n(x_0),\ t\ge 0.
    \end{cases}
    \end{equation}
    A solution $u\in L^{\infty}([0,+\infty);\ell^r(V))$ to \eqref{eq:infinite-parabolic} is called an \emph{exhaustion solution} if there exists a subsequence $n_k\to\infty$ such that $u(x,t)=\lim_{k\to\infty} u_{n_k}(x,t)$ for all $(x,t)\in V\times[0,+\infty)$, where each $u_{n_k}$ is a solution to the mixed problem \eqref{eq:finite-ball-problem} on $B_{n_k}(x_0)$.
\end{definition}
    
    The first main results of the present paper are the following.
    
    \begin{theorem}\label{thm:existence-lr}
        If $u_0\in\ell^r(V)$ with $r>1$, then the initial value problem \eqref{eq:infinite-parabolic} 
        admits an exhaustion solution $u\in L^{\infty}([0,+\infty);\ell^r(V))$.
    \end{theorem}

      Moreover, if $q(p-1)\geq 1$ and $u_0\in \ell^1(V)$, then the solution $u$ obtained in Theorem \ref{thm:existence-lr} for some $r>1$ is in the class $L^\infty([0,T);\ell^1(V))$; see Theorem \ref{thm:mass-conservation}.
    
    \begin{theorem}\label{thm:lr-comparison}
        Assume $r\ge 1$ and $q(p-1)\ge r$. 
        Let $u,v\in L^\infty([0,T);\ell^r(V))$ be two solutions of \eqref{eq:infinite-parabolic}. 
        If
        $$
        u(x,0)\le v(x,0) \quad\text{for every }x\in V,
        $$
        then
        $$
        u(x,t)\le v(x,t) \quad\text{for every }x\in V,\; 0\le t<T.
        $$
        In particular, the solution of \eqref{eq:infinite-parabolic} in 
        $L^\infty([0,T);\ell^r(V))$ with initial value $u_0\in\ell^r(V)$ is unique.
    \end{theorem}
    
    The following theorem shows that global mass conservation holds under appropriate constraints on the parameters $p,q$ and the underlying graph.
    
    \begin{theorem}
       \label{thm:mass-conservation-exhaustion-solution}
        Assume that there exists some constant $C>0$ and $d>0$ such that $G$ satisfies
        $$
        \mu_V(B_R)\le C (1+R)^d,\quad\forall R\geq 0.
        $$ 
        Further assume that $q(p-1)\geq\frac{d-p}{d}$. Let $u\in L^{\infty}([0,+\infty);\ell^1(V))$ be an exhaustion solution of \eqref{eq:infinite-parabolic} with initial data $u_0\in \ell^1(V)$. Then for all $t>0$, we have
        $$
        \sum_{x\in V} u(x,t)\,\mu(x) = \sum_{x\in V} u_0(x)\,\mu(x).
        $$
    \end{theorem}
    
    We now turn to large-time estimates for solutions on infinite graphs satisfying Faber--Krahn inequalities.   

    \begin{definition}\label{def:Faber--Krahn}
        We say that an infinite graph $G$ satisfies the \emph{Faber--Krahn inequality} for a given $p>1$ and a function 
        $$
        \Lambda_p:(0,+\infty)\to(0,+\infty)
        $$
        if for every finite subset $U\subset V$ and every function $f:V\to\mathbb{R}$ vanishing outside $U$, one has
        \begin{equation}\label{uFK}
        \Lambda_p(\mu_V(U)) \sum_{x\in V}|f(x)|^p \mu(x) \le \sum_{x,y\in V} |f(y)-f(x)|^p\,\omega_{xy}.
        \end{equation}
    \end{definition}
    
    We will always assume that $\Lambda_p(v)$ is continuous, decreasing, and that there exist constants $N$ and $\nu$ such that, for $v>0$,
    \begin{equation}\label{assumponLamb}
    v\mapsto \Lambda_p(v)^{-1} v^{-p/N} \text{ is nondecreasing}, \quad
    v\mapsto \Lambda_p(v)^{-1} v^{-\nu} \text{ is nonincreasing}.
    \end{equation}
    In \cref{sec:appendix} we will show that for a large class of graphs, one may take $\Lambda_p(v)=\gamma_0 v^{-p}$ with some constant $\gamma_0>0$.
    
    It is also convenient to introduce, for $r\geq 1$, the following function:
    $$
    \psi_{r}(s):=s^{\frac {q(p-1)-1}{r}} \Lambda_p\bigl(s^{-1}\bigr), \quad s>0,
    $$
    When $q(p-1)>1$, the assumptions in \eqref{assumponLamb} ensure that $\psi_r$ is continuous, strictly increasing, and satisfies $\lim_{s \to 0} \psi_r(s) = 0$ as well as $\lim_{s \to +\infty} \psi_r(s) = +\infty$. Consequently, $\psi_r$ is a bijection on $(0,+\infty)$, and we denote its inverse by $\varphi_r$.
    
        For $x\in V$, let us apply (\ref{uFK}) with $U=\{x\}$ and $f=\mathbf{1}_U$. Then
        $\Lambda_p(\mu(x))\mu(x) \le 2\mu(x)$. 
        Since $\Lambda_p$ is decreasing, we obtain $\mu(x)\ge \Lambda_p^{-1}(2)$, and consequently we always have
        $\mu_0=\inf_{x\in V}\mu(x)> 0$ on graphs satisfying the Faber--Krahn inequality.

    \begin{theorem}\label{thm:upper-boundint}
        Assume that $G$ satisfies the Faber--Krahn inequality (\ref{uFK}) and suppose that $q(p-1)> 1$. 
        Assume that $u_0\in \ell^r(V)$, $r\geq 1$, is nonnegative and let $u\in L^{\infty}([0,T);\ell^r(V))$ be a solution of \eqref{eq:infinite-parabolic}. 
        Then, for all $0<t<T$,
        \begin{equation}
           	||u(t)||_{\ell^{\infty}(V)}
	\le
	C
	\|u_0\|_{\ell^r(V)}
	\varphi_r
	\left(
	t^{-1}
	\|u_0\|_{\ell^r(V)}^{-\frac{q(p-1)-1}{r}}
	\right)^{1/r},
        \end{equation}
        where the constant $C>0$ depends only on $p$, $q$, $r$ and $\Lambda_p$.
    \end{theorem} 
    
    \begin{theorem}\label{thm:propagation-finite-supportint}
        Assume that $G$ satisfies the Faber--Krahn inequality (\ref{uFK}) and that $q(p-1)>1$. 
        Let $u_0\in\ell^1(V)$ be nonnegative and have its support contained in a ball $B_{R_0}$, and let 
        $u\in L^{\infty}([0,T);\ell^1(V))$ be a solution of \eqref{eq:infinite-parabolic}. 
        Then for every $0<\varepsilon<1$ there exists a constant $c>0$ such that
        $$
        \|u(t)\|_{\ell^1(B_R)} \ge (1-\varepsilon)\|u_0\|_{\ell^1(V)}
        $$
        whenever
        \begin{equation}\label{suitR}
            R \ge \max\!\left\{ 
            \frac{c}{\varepsilon}\, t^{1/p}\, \|u_0\|_{\ell^1(V)}^{\frac{q(p-1)-1}{p}}\;
            \varphi_1\!\Bigl( t^{-1} \|u_0\|_{\ell^1(V)}^{-[q(p-1)-1]} \Bigr)^{\!\frac{q(p-1)-1}{p}},\;
            2 R_0 
        \right\}.
        \end{equation}
    \end{theorem}

       By taking $\varepsilon=\frac12$ in \cref{thm:propagation-finite-supportint}, one obtains that, for $R$ as in (\ref{suitR}),
        \begin{equation}\label{eq:lower-bound}
            \|u(t)\|_{\ell^{\infty}(V)} \ge \frac{\|u_0\|_{\ell^1(V)}}{2\,\mu_V(B_R)}.
        \end{equation}
        In view of \cref{thm:lr-comparison}, \cref{thm:propagation-finite-supportint} and \eqref{eq:lower-bound} also hold for nonnegative initial data $u_0\in \ell^1(V)$.
        It is notable that Theorem \ref{thm:upper-boundint} and Theorem \ref{thm:propagation-finite-supportint} contain the corresponding results proved by Andreucci and Tedeev in \cite{Dan}.
        
	
	To state our next main result, we need the concept of $d$-isoperimetric inequality.
	For a subset $\Omega$ of $V$, define the edge boundary
	$$
	\partial\Omega:=\{\{x,y\}\in E:x\in\Omega,y\in V\backslash\Omega\}.
	$$
	We say that $G$ satisfies the $d$-isoperimetric inequality for a fixed $d>1$ if there exists a finite constant $C_d$ such that for all finite subsets $\Omega$ of $V$, 
	$$
	\mu_V(\Omega)^{\frac{d-1}{d}}\leq C_d\mu_E(\partial\Omega).$$
	\begin{theorem}\label{thm:finite-time-extinction-infinite}
		Assume that $G$ satisfies the $d$-isoperimetric inequality for some $d>p>1$ and that $q(p-1) <\frac{d-p}{d}$. Set $ m=\frac{d(1-q(p-1))}{p}>1$. Let $u \in L^{\infty}([0,+\infty);\ell^r(V))$ be a solution of \eqref{eq:infinite-parabolic} with $u_0\in\ell^{m}(V)$. If $q>1$, then there exists a constant $C(p,q,d)>0$ such that 
        $$
        u(\cdot,t)\equiv 0
        \quad\text{for all }\ t\ge C(p,q,d)\|u_0\|_{\ell^m(V)}^{mp/d}.
        $$
        Furthermore, if $u$ is an exhaustion solution, then this finite-time extinction result holds for the full range $q>0$.
	\end{theorem}
    
	When $q=1$, Theorem \ref{thm:finite-time-extinction-infinite} was proved in Theorem 1.4 in \cite{Hua1}.
	
	\begin{example}
	    As an example, we now discuss the implications of our results for Cayley graphs with polynomial growth. Let $G=(\Gamma, S)$ be the Cayley graph of a discrete group ${\Gamma}$ with a finite symmetric generating set $S$, where the weights $\omega$ satisfy $\omega_{xy}=1$ whenever $x\sim y$. Denote by $e$ the identity element of $\Gamma$ and define 
        $$ 
        B_R = \{x \in V : d (x, e) \leq R\}
        $$
        as the ball centered at e with radius $R>0$. We assume that $G=(\Gamma, S)$ has polynomial growth of order $N>1$, that is, there exists a constant $C>0$ such that
        $$
        \frac{1}{C}(1+R)^N\leq\mu_V(B_{R})\leq C(1+R)^N,\quad \forall R> 0.
        $$
        It follows from \cref{lem:Cayley-isoperimetric} that $G$ satisfies the $N$-isoperimetric inequality. By \cref{thm:Cayley-Faber--Krahn}, we can take $\Lambda_p(v)=\gamma_0 v^{-\frac{p}{N}}$ for some constant $\gamma_0>0$. It follows that $\varphi_1(s)$ is a constant multiple of $s^{\frac{N}{N[q(p-1)-1]+p}}$. 

        We first consider the finite-time extinction properties of exhaustion solutions. When $0 < q(p-1) < \frac{N-p}{N}$, \cref{thm:finite-time-extinction-infinite} shows that any exhaustion solution of \eqref{eq:infinite-parabolic} with initial data $u_0 \in \ell^1(\Gamma)$ vanishes in finite time; that is, there exists $t_0 > 0$ such that $u(x,t) = 0$ for all $x \in V$ and all $t \ge t_0$. In contrast, when $q(p-1) \ge \frac{N-p}{N}$, \cref{thm:mass-conservation-exhaustion-solution} implies that any exhaustion solution of \eqref{eq:infinite-parabolic} with non-trivial nonnegative initial data $u_0 \in \ell^1(\Gamma)$ preserves its total mass over time, and therefore cannot vanish in finite time. This establishes a sharp dichotomy in the qualitative behavior of exhaustion solutions.

        Now assume that $q(p-1)> 1$. It follows from \cref{thm:upper-boundint} that there exists a constant $C>0$ such that for each solution $u\in L^{\infty}([0,+\infty);\ell^1(V))$ of \eqref{eq:infinite-parabolic}, we have
        $$
        \|u(t)\|_{\ell^{\infty}(\Gamma)} \le C t^{-\alpha_*},\qquad\forall t>0,
        $$
        where 
        $$
        \alpha_{*} := \frac{N}{N[q(p-1)-1] + p}.
        $$
        If, in addition, the initial data $u_0\in\ell^1(V)$ is nonnegative, we obtain from \cref{thm:propagation-finite-supportint} and \eqref{eq:lower-bound} that there exists a constant $\gamma>0$ such that 
        $$
        \|u(t)\|_{\ell^{\infty}(\Gamma)} \ge \gamma (1+t)^{-\alpha_{*}},\qquad \forall t>0.
        $$
        In this sense, the upper bound obtained above is sharp.
	\end{example}

    The structure of the paper is as follows. In Section \ref{sec:existence-comparison}, we prove the existence theorem, comparison principles, and conservation of mass for solutions of \eqref{eq:infinite-parabolic}.  In Section \ref{sec:asymptotic-propagation}, we establish the asymptotic estimates stated in Theorems \ref{thm:upper-boundint} and \ref{thm:propagation-finite-supportint}. Section \ref{sec:conservation-extinction} is devoted to the finite-time extinction result Theorem \ref{thm:finite-time-extinction-infinite}. In Section \ref{sec:finite-graph-classification} we investigate solutions of (\ref{eq:infinite-parabolic}) on finite graphs.
    We prove that these solutions extinct for $q(p-1)<1$
    and remain strictly positive for $q(p-1)\geq 1$. Finally, in Appendix \ref{sec:appendix}, we construct Faber--Krahn functions for certain classes of graphs.
  

	\section{Existence, Comparison Principles and Conservation of Mass}\label{sec:existence-comparison}
	
	\subsection{Existence}\label{subsec:existence}

    \begin{lemma}\label{lem:Green-formula}
        Let $u(x,t)$ be a function on $V\times[0,T)$ and $v(x)$ be a function on $V$ with finite support. Then
        $$
        \sum_{x\in V} (\Delta_p u^q(x,t)) v(x)\mu(x)=-\frac{1}{2}\sum_{x,y\in V}(\nabla_{xy}u^q(x,t))^{p-1} \nabla_{xy} v(x)\omega_{xy}.
        $$
    \end{lemma}
    \begin{proof}
        Set $A(x, y)=(\nabla_{xy}u^q(x,t))^{p-1}
    	\omega_{xy}$. By definition, $\Delta_p u^q(x,t)=\sum_{y\in V}A(x, y)$. 
        Since $A(x, y)=-A(y, x)$, we have
        $$
        \begin{aligned}
            \sum_{x\in V} (\Delta_p u^q(x,t)) v(x)\mu(x)=\sum_{x,y\in V} A(x,y) v(x)
            &=-\frac{1}{2} 	\sum_{x,y\in V}A(x, y) \nabla_{xy}v(x)\\
            &=-\frac{1}{2}\sum_{x,y\in V}(\nabla_{xy}u^q(x,t))^{p-1} \nabla_{xy} v(x)\omega_{xy},
        \end{aligned}
        $$
        which completes the proof.
    \end{proof}
    
    \begin{proof}[Proof of Theorem \ref{thm:existence-lr}]	
    	Let $x_0$ be an arbitrary vertex in $V$. By the Peano existence theorem, the ordinary differential equation system	
    	\[
    	\begin{cases}
    		\displaystyle
    		\frac{\partial u_n}{\partial t}(x,t)
    		=
    		\Delta_p u_n^q(x,t),
    		&
    		x\in B_n(x_0),\ t>0,
    		\\[1ex]
    		u_n(x,0)=u_0(x),
    		&
    		x\in B_n(x_0),
    		\\[1ex]
    		u_n(x,t)=0,
    		&
    		x\notin B_n(x_0),\ t\ge 0,
    	\end{cases}
    	\]
    	has a solution for sufficiently small $t>0$. Let us show that the
    	solution exists for all $t>0$.
    	Let $[0,T)$ be the right-sided maximal interval of existence. If
    	$T<+\infty$, then $u_n(x,t)$ must be bounded on $[0,T)$.
    	From the ODE system, we have
    	\[
    	u_n(x,t)_{+}^{\,r-1}
    	\frac{\partial u_n}{\partial t}(x,t)\mu(x)
    	=
    	u_n(x,t)_{+}^{\,r-1}
    	\Delta_pu_n^q(x,t)\mu(x),
    	\]
    	for all $x\in V$ and $0<t<T$.
    	Summing over $x$ and integrating over $(0, t)$, we obtain
    	\begin{equation*}
    	    \frac1r\left(\sum_{x\in V}u_n(x,t)_+^r\mu(x)-\sum_{x\in V}u_0(x)_+^r\mu(x)\right)=\int_{0}^{t}\sum_{x\in V}(\nabla_{xy}u_n^q(x,s))^{p-1}u_n(x,s)_+^{\,r-1}\mu(x)ds.
    	\end{equation*}
        By \cref{lem:Green-formula}, we have
        \begin{equation}\label{partint}
            \sum_{x\in V}(\nabla_{xy}u_n^q(x,s))^{p-1}u_n(x,s)_+^{\,r-1}\mu(x)ds=	-\frac12\sum_{x,y\in V}|\nabla_{xy}u_n^q(x,s)|^{p-1}\left|\nabla_{xy}\!\left(u_n(x,s)_+^{\,r-1}\right)\right|\omega_{xy},
        \end{equation}
    	which is negative. Therefore, 	
        $$
        \sum_{x\in V}
        	u_n(x,t)_+^r\mu(x)
        	\le
        	\|u_0\|_{\ell^r(V)}^r.
        $$
        	Similarly,
        $\sum_{x\in V}
    	u_n(x,t)_-^r\mu(x)
    	\le
    	\|u_0\|_{\ell^r(V)}^r.$
    	Since $\inf_{x\in V}\mu(x)>0$	
    	we conclude that $u_n(x,t)$ is uniformly bounded with respect to
    	$(x,t)\in V\times[0,T)$ and $n\in\mathbb N$.
        Moreover, $$|\partial_tu_n(x,t)|=\frac1{\mu(x)}
\left|
\sum_{y\sim x}
\bigl(\nabla_{xy}u_n^q(x,t)\bigr)^{p-1}
\omega_{xy}
\right|
\le
\frac1{\mu(x)}
\sum_{y\sim x}
|\nabla_{xy}u_n^q(x,t)|^{p-1}\omega_{xy}
\le C,
$$
where $C>0$ is independent of $x, t$ and $n$.
        
    	Hence, 	$\{u_n(x,t)\}_{n\ge1}$	
    	is uniformly bounded and equicontinuous on every compact time interval.
    	Using a standard diagonal process, we can extract a subsequence $\{u^{(k)}_{n_k}(x,t)\}_{k\ge1}$ such that, for all $x\in V$, $u^{(k)}_{n_k}(x,t)$ converges uniformly on on every compact time interval.	
    	Let $ u(x,t)$ be the limit function of
    	$\{u^{(k)}_{n_k}(x,t)\}_{k\ge1}$.
        It remains to show that $u(x,t)$ is a solution of (\ref{eq:infinite-parabolic}).
        For every fixed vertex $x\in V$ and all sufficiently large $n$, we have
$x\in B_n(x_0)$. Hence,
$$
u_n(x,t)
=
u_0(x)
+
\int_0^t
\Delta_pu_n^q(x,s)\,ds.
$$
Since the graph is locally finite, the operator
$$
\Delta_pu_n^q(x,s)
=
\frac1{\mu(x)}
\sum_{y\sim x}
\bigl(u_n^q(y,s)-u_n^q(x,s)\bigr)^{p-1}\omega_{xy}
$$
contains only finitely many terms. Moreover, for every neighbour $y\sim x$, the
convergence
$
u_n(y,\cdot)\to u(y,\cdot)
$ implies that 
$$
\Delta_pu_n^q(x,\cdot)
\to
\Delta_pu^q(x,\cdot)
$$
uniformly on every compact time interval. Passing to the limit in the above
integral identity, we obtain
$$
u(x,t)
=
u_0(x)
+
\int_0^t
\Delta_pu^q(x,s)\,ds.
$$
Since the integrand is continuous, the fundamental theorem of calculus implies
that $u(x,\cdot)\in C^1([0,\infty))$ and
$u$ is a solution of (\ref{eq:infinite-parabolic}).
Then $ u\in	L^\infty([0,+\infty);\ell^r(V))$ follows from Fatou's Lemma. \end{proof}

\subsection{Nonnegativity and a Comparison Principle}\label{subsec:nonnegativity}

        We will frequently use the following observation.
        \begin{remark}\label{rem:fk-measure-lower}
        Recall that we assume
        $\mu_0=\inf_{x\in V}\mu(x) > 0$.
        Therefore, for any $v\in\ell^r(V)$ and any $x\in V$,
        $$
        |v(x)| \le \mu_0^{-1/r} \|v\|_{\ell^r(V)}.
        $$
        When $s>r$, this yields
        $$
        \sum_{x\in V} |v(x)|^s \mu(x) 
            \le \mu_0^{-(s-r)/r} \|v\|_{\ell^r(V)}^{s-r} 
              \sum_{x\in V} |v(x)|^r \mu(x),
        $$
        so that $\ell^r(V)\subset\ell^s(V)$ for $s>r$.
    \end{remark}
    
        \begin{theorem}\label{thm: nonneg}
    	Assume that $q(p-1)\ge1$ and let $u\in L^\infty([0,T);\ell^r(V))$ 
    	be a solution of equation (\ref{eq:infinite-parabolic}), where $r>1$ and
    	$T\in(0,+\infty]$.	
    	Then for $h\in\mathbb R\setminus\{0\}$, we have:
    		\begin{enumerate}
    		\item If $u_0\le h$, then $u(x,t)<h$ for all $t\in(0,T)$.
    		\item If $u_0\ge h$, then $u(x,t)>h$ for all $t\in(0,T)$.
    		\item If $u_0=0$, then $u(x,t)\equiv 0.$
    	\end{enumerate}
    \end{theorem}
    
    \begin{proof}	
    	Since $-u$ is a solution of equation (\ref{eq:infinite-parabolic}) with initial value $-u_0$, without loss of generality we may assume $h>0$ and only prove the first statement.
    		
    	As $\inf_{x\in V}\mu(x)>0$ and $u\in L^\infty([0,T);\ell^r(V)),$	
    	it follows from Remark \ref{rem:fk-measure-lower} that for all $l>r$, $u\in L^\infty([0,T);\ell^l(V))$.
    	Fix $x_0\in V$. For given $R_2\ge R_1+1>1$, define the cutoff
    	function $\zeta_{R_1,R_2}(x)$ on graph $G$ by	
    	\begin{equation}\label{cutoff}
    	\zeta_{R_1,R_2}(x)
    	=
    	\begin{cases}
    		1,
    		&
    		x\in B_{R_1}(x_0),
    		\\[1ex]
    		\displaystyle
    		\frac{R_2-d(x,x_0)}{R_2-R_1},
    		&
    		x\in B_{R_2}(x_0)\setminus B_{R_1}(x_0),
    		\\[2ex]
    		0,
    		&
    		x\notin B_{R_2}(x_0).
    	\end{cases}
    	\end{equation}
    	Then for $y\sim x$,	
    	\[
    	|\nabla_{xy}\zeta_{R_1,R_2}(x)|
    	=
    	|\zeta_{R_1,R_2}(y)-\zeta_{R_1,R_2}(x)|
    	\le
    	\frac1{R_2-R_1}.
    	\]
    	Let us
    	simply denote $\zeta_{R_1,R_2}(x)$ by $\zeta(x)$.
    	Multiplying both sides of the equation in (\ref{eq:infinite-parabolic}) by $\zeta(x)(u(x,t)-h)_+^{\,r-1}\mu(x)$, 	
    	summing over $x$, and integrating over $t$, we obtain as in (\ref{partint}),
    	\[
    	\frac1r
    	\sum_{x\in V}
    	\zeta(x)(u(x,\tau)-h)_+^r\mu(x)
    	-
    	\frac1r
    	\sum_{x\in V}
    	\zeta(x)(u_0(x)-h)_+^r\mu(x)
    	=
    	-\frac12(I_1+I_2),
    	\]
    	where
    	\[
    	I_1
    	=
    	\int_0^\tau
    	\sum_{x,y\in V}
    	|\nabla_{xy}(u^q(\cdot,t))|^{p-2}
    	\nabla_{xy}(u^q(\cdot,t))
    	\zeta(y)
    	\nabla_{xy}\!\left(
    	(u(x,t)-h)_+^{\,r-1}
    	\right)
    	\omega_{xy}\,dt,
    	\]
    	and
    	\[
    	I_2
    	=
    	\int_0^\tau
    	\sum_{x,y\in V}
    	|\nabla_{xy}(u^q(\cdot,t))|^{p-2}
    	\nabla_{xy}(u^q(\cdot,t))
    	\nabla_{xy}(\zeta(x))
    	(u(x,t)-h)_+^{\,r-1}
    	\omega_{xy}\,dt.
    	\]
    	
    	Now we estimate $I_1$ and $I_2$ separately.
    	Since $\nabla_{xy}(u^q(\cdot,t))
    	\nabla_{xy}\!\left(
    	(u(x,t)-h)_+^{\,r-1}
    	\right)
    	\ge0,$
    	we always have
    	$I_1\ge0$.
    	For $I_2$, we have	
    	\begin{align*}
    	|I_2|
    	&\le
    	\frac{1}{R_2-R_1}
    	\int_0^\tau
    	\sum_{x,y\in B_{R_2+1}(x_0)}
    	|\nabla_{xy}(u^q(\cdot,t))|^{p-1}
    	|u(x,t)|^{r-1}
    	\omega_{xy}\,dt\\& \leq \frac{2^{p-1}}{R_2-R_1}
    	\int_0^\tau
    	\sum_{x,y\in B_{R_2+1}(x_0)}
    	\left(
    	|u(y,t)|^{q(p-1)}
    	+
    	|u(x,t)|^{q(p-1)}
    	\right)
    	|u(x,t)|^{r-1}
    	\omega_{xy}\,dt \\&\leq \frac{2^{p-1}}{R_2-R_1}
    	\int_0^\tau
    	\left[
    	\sum_{x,y\in V}
    	|u(y,t)|^{q(p-1)}
    	|u(x,t)|^{r-1}
    	\omega_{xy}
    	+
    	\sum_{x\in V}
    	|u(x,t)|^{q(p-1)+r-1}
    	\mu(x)
    	\right]dt.
    	\end{align*}		
    	Further, we have by Hölder's inequality,
    	\[
    	\sum_{x,y\in V}
    	|u(y,t)|^{q(p-1)}
    	|u(x,t)|^{r-1}
    	\omega_{xy}
    	\le
    	\left(
    	\sum_{y\in V}
    	|u(y,t)|^{rq(p-1)}
    	\mu(y)
    	\right)^{1/r}
    	\left(
    	\sum_{x\in V}
    	|u(x,t)|^r
    	\mu(x)
    	\right)^{1-\frac1r}.
    	\]
    	Since $rq(p-1)\ge r$ and $q(p-1)+r-1\ge r$, 
    	Remark \ref{rem:fk-measure-lower} implies that both terms are bounded uniformly in $t$.
    	Therefore
    	\[
    	|I_2|
    	\le
    	\frac{C\tau}{R_2-R_1},
    	\]
    	where $C$ is independent of $\tau$, $R_1$, and $R_2$.
    	
    	Combining the estimates gives
    	\[
    	\sum_{x\in V}
    	\zeta(x)(u(x,\tau)-h)_+^r\mu(x)
    	\le
    	\frac{C\tau}{R_2-R_1}
    	+
    	\sum_{x\in V}
    	\zeta(x)(u_0(x)-h)_+^r\mu(x).
    	\]
    	Letting $R_2\to+\infty$ we deduce	
    	\[
    	\sum_{x\in B_{R_1}(x_0)}
    	(u(x,\tau)-h)_+^r\mu(x)
    	\le
    	\sum_{x\in V}
    	(u_0(x)-h)_+^r\mu(x).
    	\]
    	Thus, also letting $R_1\to+\infty$ yields,	
    	\[
    	\sum_{x\in V}
    	(u(x,\tau)-h)_+^r\mu(x)
    	\le
    	\sum_{x\in V}
    	(u_0(x)-h)_+^r\mu(x),
    	\]
    	which implies that if $u_0\le h$, then
    	$u(x,\tau)\le h$.
    	
    	Finally, we show that $u$ cannot attain $h$.
    	If there exist $x_0\in V$ and $t_0\in(0,T)$ such that $u(x_0,t_0)=h$ then necessarily $\frac{\partial u}{\partial t}(x_0,t_0)=0$.	
    	Hence, 
    	\[
    	\sum_{y\sim x_0}
    	|\nabla_{xy}(u^q)|^{p-2}
    	\nabla_{xy}(u^q)
    	\frac{\omega_{x_0y}}{\mu(x_0)}
    	=
    	0.
    	\]
    	Since $h$ is the maximum value, we conclude that $u(y,t_0)=h$  for all $y\sim x_0$.
    	By the connectedness of $G$, we obtain $u(\cdot,t_0)\equiv h>0$, which contradicts that $u(\cdot,t_0)\in\ell^r(V)$.
    	Therefore, we conclude $u(x,t)<h$ for all $x\in V$ and $t\in(0,T)$.
    \end{proof}

    \begin{remark}\label{rem:nonnegative-alternative}
         If $u_0(x)\ge 0$, setting $h_n=-\frac{1}{n}$ in \cref{thm: nonneg} and passing to the limit as $n\to\infty$ yields that $u(x,t)\ge 0$. This shows that solutions to \eqref{eq:infinite-parabolic} preserve nonnegativity. Furthermore, if there exists some vertex $x_0\in V$ and $t_0>0$ such that $u(x_0,t_0)=0$, then arguing exactly as in the final paragraph of the preceding proof, we conclude that $u(\cdot,t_0)\equiv 0$. It follows from the above theorem that $u(x,t)=0$ for all $x\in V$ and $t\geq t_0$. It is notable that this demonstrates that the finite propagation speed phenomenon, which occurs for solutions to the Leibenson equation on Riemannian manifolds when $q(p-1)>1$, is absent in the graph setting.
    \end{remark}
    
    We now prove the comparison principle for solutions of \eqref{eq:infinite-parabolic} in the class $L^\infty([0,T);\ell^r(V))$, where we assume that $q(p-1)\ge r\geq1$.
    
    \begin{proof}[Proof of \cref{thm:lr-comparison}]
    Fix $x_0\in V$. For a given $R\ge 2$, define the cutoff
	function $\zeta_{R}(x)=\zeta_{R,2R}(x)$, where $\zeta_{R_1,R_2}$ is defined as in \eqref{cutoff}. Then $0\le \zeta_R\le 1$, $\operatorname{supp}\zeta_R\subset B_{2 R}(x_0)$, and
    \begin{equation}\label{eq:cutoff-R}
        |\nabla_{xy}\zeta_R(x)|\le \frac1R
        \qquad\text{whenever }x\sim y.
    \end{equation}
    Since $\inf_{x\in V} \mu(x)>0$ and $u,v\in L^\infty([0,T);\ell^r(V))$, there exists $M>0$ such that
    \begin{equation}\label{eq:pointwise-bound}
        |u(x,t)|+|v(x,t)|\le M
        \qquad\text{for all }(x,t)\in V\times[0,T).
    \end{equation}
    Let $w=u-v$ and define
    $$
        F_R(t)=\sum_{x\in V}\zeta_R(x)w_+(x,t)\mu(x).
    $$
    Since the summand is finite, $F_R$ is absolutely continuous and
    $$
        F_R'(t)
        =
        \sum_{x\in V}\zeta_R(x)
        \bigl(\Delta_p u^q(x,t)-\Delta_p v^q(x,t)\bigr)
        \mathbf{1}_{\{w(\cdot,t)>0\}}(x)\mu(x) \quad\text{for a.e. } t
    $$
    Applying \cref{lem:Green-formula}, we obtain
    $$
        F_R'(t)=-\frac12\bigl(I_{1,R}(t)+I_{2,R}(t)\bigr),
    $$
    where
    $$
    \begin{aligned}
    I_{1,R}(t)
    &:=\sum_{x,y\in V}
    \bigl((\nabla_{xy}u^q(x,t))^{p-1}-(\nabla_{xy} v^q (x,t))^{p-1}\bigr)
    \zeta_R(y)\nabla_{xy}\mathbf{1}_{\{w(\cdot,t)>0\}}(x)\omega_{xy},\\
    I_{2,R}(t)
    &:=\sum_{x,y\in V}
    \bigl((\nabla_{xy}u^q(x,t))^{p-1}-(\nabla_{xy} v^q (x,t))^{p-1}\bigr)
    \nabla_{xy}\zeta_R(x)\mathbf{1}_{\{w(\cdot,t)>0\}}(x)\omega_{xy}.
    \end{aligned}
    $$
    
    First we show that each summand in $I_{1,R}(t)$ is nonnegative and therefore $I_{1,R}(t)\ge 0$. If $\nabla_{xy}\mathbf{1}_{\{w(\cdot,t)>0\}}(x)=0$, there is nothing to prove. Now we assume that $\nabla_{xy}\mathbf{1}_{\{w(\cdot,t)>0\}}(x)>0$. It follows that $w(y,t)>0$ and $w(x,t)\leq0$. Then, we have
    $$
        (u(y,t))^q-(v(y,t))^q>0\geq  (u(x,t))^q-(v(x,t))^q,
    $$
    which implies that
    $$
        \nabla_{xy}u^q(x,t)-\nabla_{xy} v^q (x,t)\ge 0.
    $$
    Since the function $s\mapsto s^{p-1}$ is increasing, we see that
    the corresponding summand in $I_{1,R}(t)$ is
    nonnegative. The case where $\nabla_{xy}\mathbf{1}_{\{w(\cdot,t)>0\}}(x)<0$ follows similarly.
    
    Now we estimate the term $I_{2,R}(t)$. Set $\alpha=q(p-1)$. By \eqref{eq:cutoff-R},
    $$
    \begin{aligned}
    |I_{2,R}(t)|
    &\le \frac1R
    \sum_{x,y\in B_{2R+1}(x_0)}
    \left|(\nabla_{xy}u^q(x,t))^{p-1}-(\nabla_{xy} v^q (x,t))^{p-1}\right|\omega_{xy}\\
    &\le \frac{C_p}{R}
    \sum_{x,y\in B_{2R+1}(x_0)}
    \bigl(|u(y,t)|^\alpha+|u(x,t)|^\alpha
          +|v(y,t)|^\alpha+|v(x,t)|^\alpha\bigr)\omega_{xy}\\
    &\le \frac{2C_p}{R}
    \sum_{x\in V}\bigl(|u(x,t)|^\alpha+|v(x,t)|^\alpha\bigr)\mu(x).
    \end{aligned}
    $$
    Since $\alpha\ge r$, from \eqref{eq:pointwise-bound} we obtain
    $$
        \sum_{x\in V}|u(x,t)|^\alpha\mu(x)
        \le M^{\alpha-r}\sum_{x\in V}|u(x,t)|^r\mu(x)
    $$
    and the analogous inequality holds for $v$. Therefore, there exists a constant $C$ independent of $R$
    and $t$ such that
    \begin{equation}\label{eq:i2-lr}
        |I_{2,R}(t)|\le \frac{C}{R}.
    \end{equation}
    
    Combining $I_{1,R}(t)\ge0$ and \eqref{eq:i2-lr}, we obtain
    $$
        F_R'(t)\le \frac{C}{2R}
        \qquad\text{for a.e. }t\in(0,T).
    $$
    Integrating over $(0,\tau)$ gives
    $$
        \sum_{x\in V}\zeta_R(x)(u(x,\tau)-v(x,\tau))_+\mu(x)
        \le
        \sum_{x\in V}\zeta_R(x)(u(x,0)-v(x,0))_+\mu(x)
        +\frac{C\tau}{2R}.
    $$
    Since $ u(x,0)\le v(x,0)$ for every $x\in V$, the first term on the right hand side is zero. Letting
    $R\to\infty$ and applying monotone convergence yields
    $$
        \sum_{x\in V}(u(x,\tau)-v(x,\tau))_+\mu(x)=0.
    $$
    Hence, $u(x,\tau)\le v(x,\tau)$ for every $x\in V$ and $0\le \tau<T$.
    \end{proof}

    \subsection{Conservation of mass}\label{subsec:conservation-of-mass}

    \begin{theorem}\label{thm:mass-conservation}
		Assume that $q(p-1)\geq 1$. Let $u\in L^{\infty}([0,T);\ell^r(V))$, $r\geq1$, be a solution of \eqref{eq:infinite-parabolic} with initial value $u_0\in \ell^1(V)$. Then $u\in L^{\infty}([0,T);\ell^1(V))$ and we have, for all $t\in[0,T)$,
		\begin{equation}\label{l1contraction}
		\sum_{x\in V}u(x,t)\mu(x)= \sum_{x\in V}u_0(x)\mu(x).
		\end{equation}
        Moreover, for each solution $u\in L^{\infty}([0,T);\ell^r(V))$, $r>1$ of \eqref{eq:infinite-parabolic}, we have
        \begin{equation}\label{lrcontratcion}||u(\cdot, t)||_{\ell^{r}(V)}\leq ||u_0||_{\ell^{r}(V)},\qquad\forall t\in[0,T).
        \end{equation}
	\end{theorem}
	\begin{proof}
		 We first prove that $\|u(\cdot,t)\|_{\ell^1(V)}\leq \|u_0\|_{\ell^1(V)}$, which implies that $u\in L^{\infty}([0,T);\ell^{1}(V))$.
		
		Fix $R_2\geq R_1+1>1$ and $h>0$. Let $\zeta(x)$ be the cutoff function $\zeta_{R_1,R_2}(x)$ defined in \eqref{cutoff}. Multiply both sides of \eqref{eq:infinite-parabolic} by $\frac{(u(x,t)-h)_{+}}{u(x,t)+\varepsilon}\zeta(x)\mu(x)$, similar to \cref{thm: nonneg}, sum over $x$ and then integrate over $t$, we have
		\begin{equation}\label{eq:J1J2}
		    \begin{aligned}
			&\int_{0}^{\tau}\sum_{x\in V} \zeta(x)\mu(x)\frac{(u(x,t)-h)_{+}}{u(x,t)+\varepsilon}\frac{\partial u}{\partial t}(x,t)\ { d} t\\
			=&-\frac{1}{2}\int_{0}^{\tau}\sum_{x,y\in V}\left(\nabla_{xy}(u^q(\cdot,t))\right)^{p-1} \nabla_{xy}(\zeta(x))\frac{(u(y,t)-h)_{+}}{u(y,t)+\varepsilon}\omega_{xy}\ { d} t\\
			&\quad\quad-\frac{1}{2}\int_{0}^{\tau}\sum_{x,y\in V}\left(\nabla_{xy}(u^q(\cdot,t))\right)^{p-1} \zeta(x)\nabla_{xy}\left(\frac{(u(x,t)-h)_{+}}{u(x,t)+\varepsilon}\right)\omega_{xy}\ { d} t.
		\end{aligned}
		\end{equation}
		Let $s=u(x,t)$. By a change of variable, we obtain from \eqref{eq:J1J2} the following equality:
		\begin{equation}\label{eq:mass-test-identity}
			\sum_{x\in V}\int_{0}^{u(x,\tau)}\frac{(s-h)_{+}}{s+\varepsilon}\zeta(x)\mu(x)\ { d} s-\sum_{x\in V}\int_{0}^{u_0(x)}\frac{(s-h)_{+}}{s+\varepsilon}\zeta(x)\mu(x)\ { d} s= -\frac{1}{2}(J_1+J_2),
		\end{equation}
        where
		\begin{align*}
			J_1&=\int_{0}^{\tau}\sum_{x,y\in V}\left(\nabla_{xy}(u^q(\cdot,t))\right)^{p-1} \nabla_{xy}(\zeta(x))\frac{(u(y,t)-h)_{+}}{u(y,t)+\varepsilon}\omega_{xy}\ { d} t,\\
			J_2&=\int_{0}^{\tau}\sum_{x,y\in V}\left(\nabla_{xy}(u^q(\cdot,t))\right)^{p-1} \zeta(x)\nabla_{xy}\left(\frac{(u(x,t)-h)_{+}}{u(x,t)+\varepsilon}\right)\omega_{xy}\ { d} t.
		\end{align*}
        
        Let $M_r := \sup_{0\leq t<T} \|u(t)\|_{\ell^r(V)}$. Since we assume that $\mu_0=\inf_{x\in V}\mu(x)>0$, we have
        $$
        |u(x,t)| \le \mu_0^{-1/r} M_r ,\quad\forall (x,t)\in V\times[0,T).
        $$
        It follows that there exists constant $C,C'>0$ such that
        $$
        |\nabla_{xy}(u^q)|^{p-1}
        \le C \bigl( |u(x)|^{q(p-1)} + |u(y)|^{q(p-1)} \bigr)
        \le C' M_r^{q(p-1)}.
        $$
        For fixed $h>0$, define
        $
        E_h(t) = \{ y\in V : u(y,t) > h \}.
        $
        Then
        $$
        \mu_V(E_h(t)) \le h^{-r} \|u(t)\|_{\ell^r(V)}^r \le h^{-r} M_r^r .
        $$
        Therefore, for $J_1$, we have
        $$
        \begin{aligned}
        |J_1|
        \le \frac{1}{R_2-R_1} \int_0^\tau \sum_{x,y\in V} |\nabla_{xy}(u^q)|^{p-1} {\mathbf{1}}_{E_h(t)}(y) \,\omega_{xy}\,dt 
        &\le \frac{C' M_r^{q(p-1)}}{R_2-R_1} \int_0^\tau \sum_{y\in E_h(t)} \mu(y)\,dt \\
        &\le \frac{C' M_r^{r+{q(p-1)}}\tau}{h^r(R_2-R_1)} ,
        \end{aligned}
        $$
        For $J_2$, we have $J_2\geq 0$ as $\nabla_{xy}u^q\nabla_{xy}\left(\frac{(u(x,t)-h)_{+}}{u(x,t)+\varepsilon}\right)\geq 0$.
		
		Notice that $\frac{(s-h)_{+}}{s+\varepsilon}=0$ when $s<0$. From the above estimates and \eqref{eq:mass-test-identity} we can obtain that
		$$
		\begin{aligned}
			\sum_{x\in B_{R_1}(x_0)}\int_{0}^{u(x,\tau)_{+}}\frac{(s-h)_{+}}{s+\varepsilon}\mu(x)\ { d} s&\leq\sum_{x\in V}\int_{0}^{u(x,\tau)_{+}}\frac{(s-h)_{+}}{s+\varepsilon}\zeta(x)\mu(x)\ { d} s\\
			&\leq\sum_{x\in V}\int_{0}^{u_0(x)_{+}}\frac{(s-h)_{+}}{s+\varepsilon}\zeta(x)\mu(x)\ { d} s -\frac{J_1}{2}\\
			&\leq \sum_{x\in V}u_0(x)_{+}\mu(x)-\frac{J_1}{2}.
		\end{aligned}
		$$
		Now let $R_2$ tend to $+\infty$, we have
		$$
		\begin{aligned}
			\sum_{x\in B_{R_1}(x_0)}\int_{0}^{u(x,\tau)_{+}}\frac{(s-h)_{+}}{s+\varepsilon}\mu(x)\ { d} s \leq \sum_{x\in V}u_0(x)_{+}\mu(x).
		\end{aligned}
		$$
		Then letting $R_1$ tend to $+\infty$, we can obtain that
		$$\sum_{x\in V}\int_{0}^{u(x,\tau)_{+}}\frac{(s-h)_{+}}{s+\varepsilon}\mu(x)\ { d} s \leq \sum_{x\in V}u_0(x)_{+}\mu(x).
		$$
		Finally, let $\varepsilon$ and $h$ tend to $0$, using the monotone convergence theorem, we can deduce that for all $\tau\in(0,T)$,
		$$
		\sum_{x\in V}u(x,\tau)_{+}\mu(x)\leq \sum_{x\in V}u_0(x)_{+}\mu(x).
		$$ 
		Since $-u$ is also a solution of \eqref{eq:infinite-parabolic}, we also have
		$$
		\sum_{x\in V}u(x,\tau)_{-}\mu(x)\leq \sum_{x\in V}u_0(x)_{-}\mu(x).
		$$
		It follows that $\Vert u(\cdot,t)\Vert_{\ell^1(V)}\leq \Vert u_0\Vert_{\ell^1(V)}$. 
        Multiplying both sides of \eqref{eq:infinite-parabolic} by $u^{r-1}\zeta(x)\mu(x)$ and using similar arguments as above, (\ref{lrcontratcion}) follows as well.
		
		 Now we prove that $\sum_{x\in V}u(x,t)\mu(x)= \sum_{x\in V}u_0(x)\mu(x)$.
		
		Multiply both sides of the equation \eqref{eq:infinite-parabolic} by $\zeta(x)\mu(x)$. After summing over $x$ and integrating over $t$, we obtain that
		$$
		\begin{aligned}
			&\sum_{x\in V}\zeta(x)u(x,\tau)\mu(x)-\sum_{x\in V}\zeta(x)u_0(x)\mu(x)\\
			\geq&-\frac{1}{2}\int_{0}^{\tau}\sum_{x,y\in V}|\nabla_{xy}(u^q(\cdot,t))|^{p-1} |\nabla_{xy}(\zeta(x))|\omega_{xy}\ { d} t\\
			\geq&-\frac{C}{2(R_2-R_1)} \int_{0}^{\tau}\sum_{x,y\in V}(|u(y,t)|^{q(p-1)}+|u(x,t)|^{q(p-1)})\omega_{xy}\ { d} t\\
			\geq&-\frac{C\tau}{2(R_2-R_1)\mu_0^{q(p-1)-1}} \Vert u_0\Vert_{\ell^1(V)}^{q(p-1)},
		\end{aligned}
		$$
		where the last inequality follows from the fact that
		$$
		|u(x,t)|\leq\frac{1}{\mu_0}\|u(\cdot,t)\|_{\ell^1(V)}\leq\frac{1}{\mu_0}\|u_0\|_{\ell^1(V)}.
		$$
        
		Now let $R_2,R_1$ tend to $+\infty$ successively. By the dominated convergence theorem, we obtain that
		$$
		\sum_{x\in V}u(x,\tau)\mu(x)\geq \sum_{x\in V}u_0(x)\mu(x).
		$$
		Applying the similar argument to the solution $-u$ yields 
        $$
        -\sum_{x\in V}u(x,\tau)\mu(x)\geq -\sum_{x\in V}u_0(x)\mu(x).
        $$
        Thus, for all $\tau\in(0,T)$,
		$$
		\sum_{x\in V}u(x,\tau)\mu(x)= \sum_{x\in V}u_0(x)\mu(x),
		$$ 
        which completes the proof of (\ref{l1contraction}).
	\end{proof}
    When considering exhaustion solutions, the conditions on the parameters $p$ and $q$ in \cref{thm:mass-conservation} can be further weakened under the assumption of polynomial growth of the underlying graph. We begin by establishing several necessary lemmas.
    \begin{lemma}\label{lem:power-difference}
        Let $p \geq 1$, $q > 0$, and let $m' = 1 + \frac{\varepsilon}{q}$ for some $\varepsilon \in (0,1)$. Set $s = \frac{m'+p-2}{p}$. Then, for all $a > 0$ and $b > 0$, we have
        \[
            |a \pm b|^{p-1} \left| a^{m'-1} \pm b^{m'-1} \right| \geq c_{p,q} \varepsilon \left| a^s \pm b^s \right|^p,
        \]
        where $c_{p,q} > 0$ is a constant depending only on $p$ and $q$.
    \end{lemma}
    
    \begin{proof}
        We first prove the inequality for the minus sign. Without loss of generality, we assume $a > b$. By the fundamental theorem of calculus and H\"older's inequality, there exists a constant $C_{p,q} > 0$ such that
        \[
        \begin{aligned}
            \left( \frac{a^s - b^s}{a - b} \right)^p = \left( \frac{s}{a - b} \int_{b}^{a} \xi^{s-1} \, d\xi \right)^p 
            &\leq \frac{s^p}{a - b} \int_{b}^{a} \xi^{(s-1)p} \, d\xi \\
            &= \frac{s^p}{a - b} \int_{b}^{a} \xi^{m'-2} \, d\xi \\
            &= \frac{s^p}{m'-1} \frac{a^{m'-1} - b^{m'-1}}{a - b} \leq \frac{C_{p,q}}{\varepsilon} \frac{a^{m'-1} - b^{m'-1}}{a - b},
        \end{aligned}
        \]
        where in the last inequality we used the relation $m'-1 = \frac{\varepsilon}{q}$ and the fact that $s^p$ is uniformly bounded from above by a constant depending only on $p$ and $q$ for all $\varepsilon \in (0,1)$.
        
        Next, we consider the case with the plus sign. Since $p \geq 1$ and $a, b > 0$, we have $(a+b)^{p-1} \geq \max\{a^{p-1}, b^{p-1}\}$. Therefore,
        \[
            (a+b)^{p-1}\left(a^{m'-1} + b^{m'-1}\right) \geq a^{p-1} \cdot a^{m'-1} + b^{p-1} \cdot b^{m'-1} = a^{m'+p-2} + b^{m'+p-2}.
        \]
        Note that $sp = m'+p-2$. By applying the elementary inequality $(x+y)^p \leq 2^{p-1}(x^p+y^p)$ for $p \geq 1$ with $x = a^s$ and $y = b^s$, we obtain
        \[
            a^{m'+p-2} + b^{m'+p-2} = a^{sp} + b^{sp} \geq 2^{-p} \left(a^s + b^s\right)^p.
        \]
        Since $\varepsilon \in (0,1)$, the above inequalities implies that
        \[
            (a+b)^{p-1}\left(a^{m'-1} + b^{m'-1}\right) \geq 2^{-p} \left(a^s + b^s\right)^p \geq 2^{-p} \varepsilon \left(a^s + b^s\right)^p.
        \]
        The conclusion then follows by choosing $c_{p,q} := \min\left\{ C_{p,q}^{-1}, 2^{-p} \right\}$.
    \end{proof}

    \begin{lemma}\label{lem:beta-energy}
        Let $\varepsilon\in(0,1)$ and set
        $
          \beta = \frac{q(p-1)+\varepsilon}{p} > 0 .
        $
        If $u\in L^\infty([0,+\infty);\ell^1(V))$ is an exhaustion solution of \eqref{eq:infinite-parabolic} with initial data $u_0\in\ell^1(V)$, then for every $t\in[0,+\infty)$, we have
        \begin{equation}\label{eq:beta-energy}
          \int_0^t \sum_{x,y\in V}
          \left\lvert u(y,\tau)^{\beta} - u(x,\tau)^{\beta} \right\rvert^{p}
          \omega_{xy}\,d\tau
          \leq\frac{C}{\varepsilon},
        \end{equation}
        where $C>0$ is a constant depend only on $p,q,\mu_0$ and $u_0$.
    \end{lemma}
    
    \begin{proof}
    Fix $x_0\in V$. Without loss of generality, we may assume that $u(x,t)=\lim_{n\to\infty} u_{n}(x,t)$ for all $(x,t)\in V\times[0,+\infty)$, where each $u_n:V\times[0,\infty)\to \mathbb{R}$ is a solution of the mixed problem \eqref{eq:finite-ball-problem}. For each $n\in\mathbb{N}$, we multiply both sides of the equation in \eqref{eq:finite-ball-problem} by $u_n^{\varepsilon}$ and sum over $x\in V$. It follows that
    \begin{equation}\label{eq:m-derivatice}
        \frac{d}{dt} \left(\frac{1}{1+\varepsilon} \sum_{x\in V} \left\lvert u_n(x,t) \right\rvert^{1+\varepsilon} \mu(x)\right)
          = - \frac{1}{2} \sum_{x,y\in V}
          \left( \nabla_{xy} (u_n^{q})(x,t) \right)^{p-1}
          \nabla_{xy} (u_n^{\varepsilon})(x,t) \omega_{xy}.
    \end{equation}
    Let $m'=1+\frac{\varepsilon}{q}$ and $s=\frac{m'+p-2}{p}$. Applying \cref{lem:power-difference} to
    $$
    a=(u_n(y,t))^{q}\quad\text{and}\quad b=(u_n(x,t))^q,
    $$
    we obtain
    $$
      \left( \nabla_{xy} (u_n^{q}) \right)^{p-1} \, \nabla_{xy} (u_n^{q(m'-1)})
      \ge c_{p,q}\varepsilon
      \left\lvert \nabla_{xy} (u_n^{qs}) \right\rvert^p.
    $$
    Since $q(m'-1)=\varepsilon$ and $qs=\beta$, it then follows from \eqref{eq:m-derivatice} that 
    $$
    \begin{aligned}
        \int_0^t \sum_{x,y\in V}
          \left\lvert \nabla_{xy} (u_n^{\beta}) \right\rvert^{p}
          \omega_{xy}\,d\tau
          &\leq\frac{1}{c_{p,q}\varepsilon} \int_0^t\sum_{x,y\in V} \left( \nabla_{xy} (u_n^{q}) \right)^{p-1} \, \nabla_{xy} (u_n^{\varepsilon})\omega_{xy}\,d\tau\\
          &=\frac{2}{(1+\varepsilon)c_{p,q}\varepsilon} \sum_{x\in V} \left(\lvert u_n(x,0) \rvert^{1+\varepsilon} -\lvert u_n(x,t) \rvert^{1+\varepsilon}\right)\mu(x)\\
          &\leq\frac{2}{c_{p,q}\varepsilon} \sum_{x\in V} \lvert u_0(x) \rvert^{1+\varepsilon} \mu(x)\leq\frac{2}{c_{p,q}\mu_0^{\varepsilon}\varepsilon} \|u_0\|_{\ell^1(V)}^{1+\varepsilon},
    \end{aligned}
   $$
   where the last inequality follows from \cref{rem:fk-measure-lower}. Finally, Fatou's lemma and the pointwise convergence $u_n\to u$ give \eqref{eq:beta-energy}. This completes the proof.
    \end{proof}

    \begin{lemma}\label{lem:power-lipschitz}
		For any real numbers $a,b$ and any $p\geq 1$, we have 
		$$
		|a^p-b^p|\leq p|a-b|(|a|^{p-1}+|b|^{p-1}).
		$$
	\end{lemma}
	\begin{proof}
		When $ab\leq 0$, the inequality is easily verified. When $ab>0$, without loss of generality, we may assume both of them are positive. Then by the mean value theorem, there exists a number $\eta$ between $a$ and $b$ such that 
        $$
        |a^p-b^p|=p|a-b||\eta|^{p-1}\leq p|a-b|(|a|^{p-1}+|b|^{p-1}).
        $$
	\end{proof}

    \begin{corollary}\label{cor:p-q-beta-inequality}
         Suppose $p\geq1$ and $0 < \beta \le q$. Then there exists a constant $C = C(p,q,\beta)$ such that for all $a,b \in \mathbb{R}$,
        \begin{equation}\label{eq:chain}
          \left| a^{q} - b^{q} \right|^{p-1}
          \le
          C\bigl( \left| a \right|^{(q-\beta)(p-1)}
          + \left| b \right|^{(q-\beta)(p-1)} \bigr)
          \left| a^{\beta} - b^{\beta} \right|^{p-1}.
        \end{equation}
    \end{corollary}
    \begin{proof}
        By \cref{lem:power-lipschitz}, we have
        $$
        |a^{q} - b^{q}|=\left|(a^{\beta})^{\frac q \beta}-(b^{\beta})^{\frac q \beta}\right|\leq \frac{q}{\beta}\bigl( \left| a \right|^{q-\beta}
          + \left| b \right|^{q-\beta} \bigr) |a^{\beta} - b^{\beta}|.
        $$
        The result then follows from the fact that
        $$
        (A+B)^{p-1} \le 2^{p-1} (A^{p-1} + B^{p-1})\quad \forall A,B\geq 0.
        $$
    \end{proof}


    \begin{proof}[Proof of \cref{thm:mass-conservation-exhaustion-solution}]
    Let $a=q(p-1)$. By assumption, we have $a\geq \frac{d-p}{d}$. In view of \cref{thm:mass-conservation}, we may assume that $a<1$. For all sufficiently large $R$, define
    \[
      \beta_R:=\frac{a+(\log R)^{-1}}{p},\quad r_R:=a-\frac{p-1}{\log R},\quad m_R:=1+\frac{1}{\log R}.
    \]
    Then $0<r_R<a<1$ and $0<\beta_R<q$. By construction, these parameters satisfy the identity
    \begin{equation}\label{eq:exponent-identity}
        p(q-\beta_R)(p-1) = a-\frac{p-1}{\log R} = r_R.
    \end{equation}
    Applying \cref{lem:beta-energy} with $\epsilon=(\log R)^{-1}$ gives
    \begin{equation}\label{eq:grad-beta-bound}
            \left( \int_0^t \sum_{x,y\in V} |\nabla_{xy}u^{\beta_R}|^p \omega_{xy} \, d\tau \right)^{\frac{p-1}{p}} \le C(\log R)^{\frac{p-1}{p}}.
    \end{equation}
        
        Fix a vertex $x_0 \in V$ and set $\rho(x) := 1+d_G(x,x_0)$. Let
        \[
            \eta(s) =
            \begin{cases}
                1, & s \le 1, \\
                2-s, & 1 < s < 2, \\
                0, & s \ge 2,
            \end{cases}
        \]
        and define the cut-off function $\zeta_R$ on $V$ by
        \[
            \zeta_R(x) := \eta\left( \frac{\log \rho(x)}{\log R} \right).
        \]
        Clearly, $0 \le \zeta_R \le 1$, $\zeta_R(x)=1$ when $\rho(x) \le R$, and $\zeta_R(x)=0$ when $\rho(x) \ge R^2$. Moreover, $\zeta_R(x) \to 1$ for every $x \in V$ as $R \to \infty$.
        Since $|\rho(x)-\rho(y)| \le 1$ whenever $x \sim y$, the mean value theorem implies that there exists a constant $C > 0$ independent of $R$ such that for any edge $\{x,y\} \in E$,
        \[
            |\nabla_{xy}\zeta_R| \le h_R(x), \quad \text{where} \quad h_R(x) := \frac{C}{\rho(x)\log R} \mathbf{1}_{\{R/2 \le \rho(x) \le 2R^2\}}.
        \]
        Now, setting $\sigma_R := \frac{p}{1-r_R}$, we compute
        \[
            1-r_R = 1-a+\frac{p-1}{\log R} \leq \frac{p}{d} + \frac{p-1}{\log R}.
        \]
        It follows that $\sigma_R\geq\frac{p}{p/d+(p-1)/\log R}$, and we have the uniform bound
        \begin{equation}\label{eq:d-sigma-diff}
             d-\sigma_R \leq\frac{d^2(p-1)}{p\log R+d(p-1)} \le \frac{C}{\log R}.
        \end{equation}
        
        We first show that there exists $C>0$ such that for all $R>2$
        \begin{equation}\label{eq:capacity-bound}
            \sum_{x\in V} h_R(x)^{\sigma_R} \mu(x) \le  C (\log R)^{1-\sigma_R}.
        \end{equation}
        Let $K_R := \lceil \log_2(4R) \rceil$, and define
        \[
            A_k := \{ x \in V : 2^{k-1}R \le \rho(x) < 2^k R \}, \quad k \in \{0, 1, \dots, K_R\}.
        \]
        Then, $\{A_k\}_{k=0}^{K_R}$ covers the support $\{x \in V : R/2 \le \rho(x) \le 2R^2\}$. Since 
        \[
        \rho(x)^{-\sigma_R} \le 2^{\sigma_R}(2^k R)^{-\sigma_R},\quad \forall  x \in A_k.
        \]
        the polynomial volume growth assumption $\mu_V(B_r) \le C r^d$ yields
        \[
            \sum_{x \in A_k} \rho(x)^{-\sigma_R} \mu(x) \le 2^{\sigma_R} (2^k R)^{-\sigma_R} \mu(A_k) \le C (2^k R)^{d-\sigma_R}.
        \]
        Notice that $2^k R \le 8R^2$ for all $0 \le k \le K_R$. In view of the upper bound on $d-\sigma_R$ from \eqref{eq:d-sigma-diff}, we have
        \[
            (2^k R)^{d-\sigma_R} \le (8R^2)^{\frac{C}{\log R}} \le C',
        \]
        where $C' > 0$ is a constant independent of $R$. Summing over all $k \in \{0, 1, \dots, K_R\}$ yields
        \[
            \sum_{R/2 \le \rho(x) \le 2R^2} \rho(x)^{-\sigma_R} \mu(x) \le \sum_{k=0}^{K_R} \sum_{x \in A_k} \rho(x)^{-\sigma_R} \mu(x) \le \sum_{k=0}^{K_R} C (2^k R)^{d-\sigma_R} \le C \log R.
        \]
        Thus, by definition we have
        \[
            \sum_{x\in V} h_R(x)^{\sigma_R} \mu(x) \le \frac{C}{(\log R)^{\sigma_R}} \sum_{R/2 \le \rho(x) \le 2R^2} \rho(x)^{-\sigma_R} \mu(x)\leq C (\log R)^{1-\sigma_R}.
        \]
        
        Next, set
        \[
            S_R(t) := \sum_{x\in V} u(x,t)\zeta_R(x) \mu(x).
        \]
        It follows from \cref{lem:Green-formula} that
        \begin{equation}\label{eq:green-formula}
            \begin{aligned}
                |S_R(t) - S_R(0)| &= \left| \int_0^t \sum_{x\in V} \frac{\partial u}{\partial \tau}(x,\tau)\zeta_R(x) \mu(x) \, d\tau \right| \\
                &\leq \int_0^t \sum_{x,y\in V} |\nabla_{xy} (u^{q})(x,\tau)|^{p-1} |\nabla_{xy}\zeta_R(x)| \omega_{xy} \, d\tau \\
                &\leq \int_0^t \sum_{x,y\in V} |\nabla_{xy} (u^{q})(x,\tau)|^{p-1} h_R(x) \omega_{xy} \, d\tau.
            \end{aligned}
        \end{equation}
        By \eqref{eq:exponent-identity} and \cref{cor:p-q-beta-inequality}, we obtain
        \begin{equation}\label{eq:chain-ineq}
            |\nabla_{xy}u^q|^{p-1} \le C \left( |u(x,\tau)|^{\frac{r_R}{p}} + |u(y,\tau)|^{\frac{r_R}{p}} \right) |\nabla_{xy}u^{\beta_R}|^{p-1},
        \end{equation}
        where the constant $C > 0$ can be chosen to be independent of $R$, since $\beta_R$ lies within a compact subinterval of $(0,q)$ for sufficiently large $R$.
        An application of H\"older's inequality to \eqref{eq:green-formula}, together with the estimate in \eqref{eq:grad-beta-bound}, yields
        \begin{equation}\label{eq:SR-bound}
            |S_R(t)-S_R(0)| \le C (\log R)^{\frac{p-1}{p}} J_R^{\frac{1}{p}},
        \end{equation}
        where
        \[
            \begin{aligned}
                J_R &:= \int_0^t \sum_{x,y\in V} \left( |u(x,\tau)|^{r_R} + |u(y,\tau)|^{r_R} \right) h_R(x)^p \omega_{xy} \, d\tau \\
                &\leq C \int_0^t \left( \sum_{x\in V} |u(x,\tau)|\mu(x) \right)^{r_R} \left( \sum_{x\in V} h_R(x)^{\frac{p}{1-r_R}}\mu(x) \right)^{1-r_R} \, d\tau.
            \end{aligned}
        \]
        
        By the estimate of $h_R(x)$ in \eqref{eq:capacity-bound} and the identity $\frac{1-r_R}{p} = \frac{1}{\sigma_R}$, we obtain
        \begin{equation}\label{eq:JR-bound}
            J_R^{\frac{1}{p}} \le C_t \left( (\log R)^{1-\sigma_R} \right)^{\frac{1-r_R}{p}} = C_t (\log R)^{\frac{1}{\sigma_R}-1}.
        \end{equation}
        Substituting \eqref{eq:JR-bound} back into \eqref{eq:SR-bound}, we obtain
        \[
            |S_R(t)-S_R(0)| \leq C_t (\log R)^{\frac{p-1}{p}+\frac{1}{\sigma_R}-1} = C_t (\log R)^{\frac{p-1}{p}+\frac{1-r_R}{p}-1} = C_t (\log R)^{-\frac{r_R}{p}},
        \]
        where $C_t>0$ is some constant dependent on $t$.
        Since $r_R = a - \frac{p-1}{\log R} \to a > 0$ as $R \to \infty$, we have $(\log R)^{-r_R/p} \to 0$ as $R \to \infty$, which implies that
        \[
            \lim_{R\to\infty} |S_R(t)-S_R(0)| = 0.
        \]
        An application of the dominated convergence theorem then yields
        \[
            \sum_{x\in V} u(x,t)\mu(x) = \sum_{x\in V} u_0(x)\mu(x).
        \]
        This completes the proof of the theorem.
    \end{proof}

	\section{Asymptotic Properties and Propagation of Finite Support Initial Data}
    \label{sec:asymptotic-propagation}
    From now on, let us use the notation $\delta:=q(p-1)-1$.
     \subsection{Asymptotic Properties}
    
   \begin{lemma}\label{monolem}
Assume that $q\geq 1$. For some real $\lambda>1$ assume that $\alpha=\frac{\lambda+\delta}{p}>0$ and let $$u:V\to\mathbb R.$$
	Then, for all $h\geq 0$ and all
	$x,y\in V$,
	\begin{equation}\label{monotone}\Bigl(
	|\nabla_{xy}u(x)^{q}|^{p-2}\nabla_{xy}u(x)^{q}
	\Bigr)
	\nabla_{xy}\!\Bigl(
	(u(x)-h)_+^{\lambda-1}
	\Bigr)\ge
	C
	\left|
	\nabla_{xy}
	\Bigl(
	(u(x)-h)_+^{\alpha}
	\Bigr)
	\right|^p ,\end{equation} where $C$ is a positive constant depending on $p, q$ and $\lambda$.
\end{lemma}

\begin{proof}
If $u(x),u(y)\leq h$, then both sides of \eqref{monotone} vanish. Now let us consider the case when $u(y)>u(x)>h$. Since $q\geq 1$, we have 
$$
u(y)^q-u(x)^q\geq (u(y)-h)^{q}-(u(x)-h)^{q}.
$$ 
Thus, with $a=u(y)-h$ and $b=u(x)-h$, it suffices to prove
\begin{equation}\label{eleineq}(a^{q}-b^{q})^{p-1}(a^{\lambda-1}-b^{\lambda-1})\geq C(a^{\alpha}-b^{\alpha})^{p}\end{equation} by setting $a=u(y)-h$ and $b=u(x)-h$. 
Indeed, we have by Hölder's inequality \begin{align*}
(a^{\alpha}-b^{\alpha})^{p}=\left(\alpha\int_{b}^{a}s^{\alpha-1}\right)^{p}&=\alpha^p\left(\int_{b}^{a}s^{\frac{(q-1)(p-1)}{p}}s^{\frac{\lambda-2}{p}}ds\right)^{p}\\&\leq \alpha^p\left(\int_{b}^{a}s^{q-1}ds\right)^{p-1}\left(\int_{b}^{a}s^{\lambda-2}ds\right)\\&=\frac{\alpha^{p}}{q^{p-1}(\lambda-1)}(a^{q}-b^{q})^{p-1}(a^{\lambda-1}-b^{\lambda-1}),
\end{align*} which implies (\ref{eleineq}). In the remaining case when $u(x)<h<u(y)$, note that since $q\geq 1$, 
$$
u(y)^q-u(x)^q\geq (u(y)-h)^{q},
$$ 
which again yields (\ref{eleineq}) and thus finishes the proof. 
\end{proof}
    
    From the above lemma, we obtain the following Caccioppoli-type inequality.
    
    \begin{lemma}\label{lemcacc}
        Suppose that $q\geq1$, $\delta=q(p-1)-1>0$, and $\lambda>1$.
        Let $u\in L^\infty([0,T);\ell^\lambda(V))$ be a solution of (\ref{eq:infinite-parabolic}) and let $\eta(t)$ be a Lipschitz function in $(0, \infty)$. Set $\alpha=\frac{\sigma}{p}$, where $\sigma=\lambda+\delta>0$. Choose $0\leq t_{1}<t_{2}< T$. Then, for any $h\geq0$,	
    	\begin{align}\nonumber
    			\left[\sum_{x\in V}
    		(u(x,t)-h)_+^\lambda\eta(t)\mu(x)\right]_{t_1}^{t_2}&+	\int_{t_1}^{t_2}
    		\sum_{x,y\in V}
    		\left|
    		\nabla_{xy}
    		\Bigl(
    		(u(x,t)-h)_+^{\alpha}
    		\Bigr)
    		\right|^p\eta(t)
    		\omega_{xy}\,dt\\&\label{caccio}\quad\leq 
    	C\int_{t_1}^{t_2}
    		\sum_{x\in V}
    		(u(x,t)-h)_+^\lambda\mu(x)|\eta^{\prime}(t)|\,dt,
    	\end{align}
    	where $C$ depends only on $p$, $q$, and $\lambda$.
    \end{lemma}
    
    \begin{proof}
    Multiplying the equation (\ref{eq:infinite-parabolic}) by
    $\zeta(x)\eta(t)
    	(u(x,t)-h)_+^{\,\lambda-1}
    	\mu(x)$ where $\zeta$ is given by (\ref{cutoff}), integrating in time over $(t_1, t_2)$ 
    	and summing over vertices gives
    \begin{align*}
    \frac1\lambda\left[\sum_{x\in V}
    (u(x,t)-h)_+^\lambda\zeta(x)\eta(t)\mu(x)\right]_{t_1}^{t_2}-	\frac1\lambda
    \int_{t_{1}}^{t_{2}}
    \sum_{x\in V}
    \zeta(x)
    (u(x,t)-h)_+^\lambda
    \mu(x)\eta'(t)\,dt=-\frac12(L_1+L_2),
    \end{align*}	
    	where
    	\[
    	L_1
    	=
    	\int_{t_1}^{t_2}
    	\sum_{x,y\in V}
    	|\nabla_{xy}(u^q(\cdot,t))|^{p-2}
    	\nabla_{xy}(u^q(\cdot,t))
    	\zeta(y)
    	\nabla_{xy}\!\Bigl(
    	(u(x,t)-h)_+^{\,\lambda-1}
    	\Bigr)
    	\omega_{xy}
    	\eta(t)\,dt,
    	\]
    	and
    	\[
    	L_2
    	=
    	\int_{t_1}^{t_2}
    	\sum_{x,y\in V}
    	|\nabla_{xy}(u^q(\cdot,t))|^{p-2}
    	\nabla_{xy}(u^q(\cdot,t))
    	\nabla_{xy}\zeta(x)
    	(u(x,t)-h)_+^{\,\lambda-1}
    	\omega_{xy}
    	\eta(t)\,dt.
    	\]
    	As in the estimate of $I_2$ from \cref{thm: nonneg}, we have	
    	\[
    	|L_2|
    	\le
    	\frac{C}{R_2-R_1}.
    	\]
    	It follows from Lemma \ref{monolem}, that there exists a constant $c>0$ such that
    	\[
    	L_1
    	\ge
    	c
    	\int_{t_1}^{t_{2}}
    	\sum_{x,y\in B_{R_1}(x_0)}
    	\left|
    	\nabla_{xy}
    	\Bigl(
    	(u(x,t)-h)_+^{\alpha}
    	\Bigr)
    	\right|^p\eta(t)
    	\omega_{xy}\,dt.
    	\]
    	Hence,
    	
    	\begin{align*}
    	\frac1\lambda\left[\sum_{x\in V}
    	(u(x,t)-h)_+^\lambda\zeta(x)\eta(t)\mu(x)\right]_{t_1}^{t_2}&+	c
    	\int_{t_1}^{t_{2}}
    	\sum_{x,y\in B_{R_1}(x_0)}
    	\left|
    	\nabla_{xy}
    	\Bigl(
    	(u(x,t)-h)_+^{\alpha}
    	\Bigr)
    	\right|^p
    	\omega_{xy}\,dt\\&\leq	\frac1\lambda
    	\int_{t_{1}}^{t_{2}}
    	\sum_{x\in V}
    	\zeta(x)
    	(u(x,t)-h)_+^\lambda
    	\mu(x)|\eta'(t)|\,dt +	\frac{C}{R_2-R_1}
    	\end{align*}
    	Letting first
    	$
    	R_2\to\infty,
    	$
    	and then
    	$
    	R_1\to\infty,
    	$
    	gives the desired estimate.	
    \end{proof}

    The next result contains Theorem \ref{thm:upper-boundint} from Introduction. 
    
    \begin{theorem}\label{thm:upper-bound}
	Assume graph $G$ satisfies the Faber--Krahn inequality and that $\delta> 0$. Let $u_0\in\ell^r(V)$, $r\geq1$, be nonnegative, and let
	$u\in L^\infty([0,T);\ell^r(V))$
    be a solution of (\ref{eq:infinite-parabolic}).
	Then, for every $0<t<T$,
		\begin{equation}\label{MV}
	||u(t)||_{\ell^{\infty}(V)}
	\le
	C
	\|u_0\|_{\ell^r(V)}
	\varphi_r
	\left(
	t^{-1}
	\|u_0\|_{\ell^r(V)}^{-\frac{\delta}{r}}
	\right)^{1/r},
\end{equation}
where the constant $C>0$ depends only on $p, q, r$ and $\Lambda_{p}$.
\end{theorem}

\begin{proof}
By assumption and by Remark \ref{rem:fk-measure-lower}, $u\in L^\infty(0,T;\ell^\lambda(V))$
for some $\lambda>1$ and, for all $k>0$, the truncated function
$(u(x, t)-k)_+$ is finitely supported.

We first consider the case $q\geq1$. 
For given $0<\rho_1<\rho_2<1/2$, $k>0$, $0<t<T$ define for $i=0,1,2,\ldots$, the
decreasing sequences
\begin{equation}\label{ki}
k_i=k\Bigl(1-\rho_2+2^{-i}(\rho_2-\rho_1)\Bigr),
\end{equation}
\begin{equation}\label{ti}
t_i=\frac t2
\Bigl(1-\rho_2+2^{-i}(\rho_2-\rho_1)\Bigr),
,
\end{equation}
and let
$
v_i(x,\tau)=(u(x,\tau)-k_i)_+.
$
Let $\sigma=\lambda+\delta$, $\alpha=\frac{\sigma}{p}$ and set
$$
m_i(\tau):=\mu_V\bigl(\{x\in V:u(x,\tau)>k_i\}\bigr).
$$
Since $\alpha p=\sigma$ and
$\operatorname{supp}v_{i+1}(\cdot,\tau)
\subset\{u(\cdot,\tau)>k_{i+1}\}$, the Faber--Krahn inequality
\eqref{uFK}, applied to $f=v_{i+1}(\cdot,\tau)^\alpha$, implies that
$$
\Lambda_p(m_{i+1}(\tau))
\sum_{x\in V}v_{i+1}(x,\tau)^\sigma\mu(x)
\leq
\sum_{x,y\in V}|\nabla_{xy}v_{i+1}(x,\tau)^\alpha|^p\omega_{xy}.
$$
Moreover, H\"older's inequality implies that
$$
\sum_{x\in V}v_{i+1}(x,\tau)^\lambda\mu(x)
\leq m_{i+1}(\tau)^{1-\frac{\lambda}{\sigma}}
\left(\sum_{x\in V}v_{i+1}(x,\tau)^\sigma\mu(x)\right)^{\frac{\lambda}{\sigma}}.
$$
Combining the two
inequalities above and then applying Young's inequality, we obtain that
\begin{align}\nonumber
\sum_{x\in V}v_{i+1}(x,\tau)^\lambda\,\mu(x)
&\le
m_{i+1}(\tau)^{1-\frac \lambda\sigma}
\Lambda_p(m_{i+1}(\tau))^{-\frac \lambda\sigma}
\left(\sum_{x,y\in V}|\nabla_{xy}v_{i+1}(x,\tau)^{\alpha}|^p\omega_{x,y}\right)^{\frac \lambda\sigma}
\\&\label{appFK}\leq\
\varepsilon^{\frac \sigma\lambda}\sum_{x,y\in V}|\nabla_{xy}v_{i+1}(x,\tau)^{\alpha}|^p\omega_{x,y}
+
\varepsilon^{-\frac \sigma\delta}
\Lambda_p(m_{i+1}(\tau))^{-\frac \lambda\delta}
m_{i+1}(\tau),
\end{align}
where $\varepsilon>0$ will be chosen below.
Integrating (\ref{appFK}) over $(t_{i+1},t)$ gives
\begin{equation}
    \begin{aligned}\label{afterinte}
    \int_{t_{i+1}}^t\sum_{x\in V}v_{i+1}(x,\tau)^\lambda\,\mu(x)d\tau
    &\leq \varepsilon^{\frac \sigma\lambda}	\int_{t_{i+1}}^t\sum_{x,y\in V}|\nabla_{xy}v_{i+1}(x,\tau)^{\alpha}|^p\omega_{x,y}\,d\tau
    \\
    &\quad+\varepsilon^{-\frac \sigma\delta}
    t\Lambda_p(M_{i+1})^{-\frac \lambda\delta}
    M_{i+1},
    \end{aligned}
\end{equation}
where $M_{i}=\sup_{t_{i}<\tau<t}m_{i}(\tau)$.
Since $q \geq 1$, we can apply \eqref{caccio} using a piecewise linear cutoff function $\eta$. Specifically, we define
\begin{equation*}
	\eta (\tau) =
	\begin{cases}
		1, & \tau \geq t_{i}, \\
		\frac{\tau - t_{i+1}}{t_{i} - t_{i+1}}, & t_{i+1} \leq \tau \leq t_{i}, \\
		0, & \tau \leq t_{i+1},
	\end{cases}
\end{equation*}
so that almost everywhere it holds
$
	|\eta'(\tau)| \leq \frac{1}{t_{i} - t_{i+1}}.
$
Together with (\ref{afterinte}) we obtain
\begin{align*}
J_{i}&:=	\sup_{t_i<\tau<t}
\sum_{x\in V}v_i(x,\tau)^\lambda\,\mu(x)+	\int_{t_i}^t\sum_{x,y\in V}|\nabla_{xy}v_i(x,\tau)^{\alpha}|^p\omega_{x,y}\,d\tau\\
&\leq 	\frac{C2^i}{t(\rho_2-\rho_1)}
\int_{t_{i+1}}^t
\sum_{x\in V}v_{i+1}(x,\tau)^\lambda\,\mu(x)\,d\tau\\
&\leq \frac{C2^i}{t(\rho_2-\rho_1)}\varepsilon^{\frac \sigma\lambda}	\int_{t_{i+1}}^t\sum_{x,y\in V}|\nabla_{xy}v_{i+1}(x,\tau)^{\alpha}|^p\omega_{x,y}\,d\tau
+\frac{C2^i}{t(\rho_2-\rho_1)}
\varepsilon^{-\frac \sigma\delta}
t\Lambda_p(M_{i+1})^{-\frac \lambda\delta}
M_{i+1}
\end{align*}
Let us set $\gamma=
\frac{C2^i}{t(\rho_2-\rho_1)}
\varepsilon^{\frac \sigma\lambda}$, that is, 
$\varepsilon=C
	\gamma^{\frac \lambda\sigma}
	t^{\frac \lambda\sigma}
	(\rho_2-\rho_1)^{\frac \lambda\sigma}
	2^{-\frac \lambda\sigma\,i}.
$
Substituting this choice of $\varepsilon$ yields
	\begin{equation}\label{comparison}
	J_i
	\le
	\gamma J_{i+1}
	+
	C
	2^{\frac{\sigma}{\delta}i}
	(\rho_2-\rho_1)^{-\frac \sigma\delta}
	\gamma^{-\frac \lambda\delta}
	t^{-\frac \lambda\delta}
	\Lambda_p(M_\infty)^{-\frac \lambda\delta}
	M_\infty,\end{equation}
	where
	$
	M_\infty
	=
	\sup_{\frac t2(1-\rho_2)<\tau<t}
	\mu_V
	\bigl(
	\{x\in V:\ u(x,\tau)>	k(1-\rho_2)\}
	\bigr).
	$
	Iterating (\ref{comparison}) gives
	\[
	J_0
	\le
	\gamma^jJ_j
	+
	\left(\sum_{i=0}^j
	\gamma^i
	2^{\frac\sigma\delta i}\right)
	\gamma^{-\frac \lambda \delta}
	(\rho_2-\rho_1)^{-\frac \sigma\delta}
	t^{-\frac \lambda\delta}
	\Lambda_p(M_\infty)^{-\frac \lambda\delta}
	M_\infty.
	\]
	Let $t_\infty := \frac{t}{2}(1-\rho_2)$ and $k_\infty := k(1-\rho_2)$. Applying \cref{lemcacc} with $h=k_j$ and a time cutoff function that is identically $1$ on $[t_\infty, t]$, we deduce that $\sup_j J_j < \infty$. Consequently, $\gamma^j J_j \to 0$ as $j \to \infty$.
    Thus, passing to the limit as $j \to \infty$ under the condition $\gamma < 2^{-\sigma/\delta}$, we obtain
    \begin{equation}\label{firstit}
    	\sup_{\frac t2(1-\rho_1)<\tau<t}
    	\sum_{x\in V}
    	(u(x,\tau)-k(1-\rho_1))_+^\lambda\,\mu(x)
    	\le
    	C
    	(\rho_2-\rho_1)^{-\frac \sigma\delta}
    	t^{-\frac \lambda\delta}
    	\Lambda_p(M_\infty)^{-\frac \lambda\delta}
    	M_\infty.
    \end{equation}

	Let us now use a second process of iteration
	based on (\ref{firstit}).
	Let $k>0$, and define, for $n\geq 0$,
	$
	\tau_n=\frac t2(1- 2^{-n-2}), 
    $
    $
	k_n=k(1- 2^{-n-2}),
	$
	$
	\bar k_n
	=
	\frac{k_n+k_{n+1}}2
	=
	k\left(1-3\cdot 2^{-n-4}\right)
	$
	and
	$$
	Y_n
	=
	\sup_{\tau_n<\tau<t}
	\mu_V
	\bigl(
	\{x\in V:\ u(x,\tau)>k_n\}
	\bigr).
	$$
	From Chebyshev's inequality, we deduce
	\[
	Y_{n+1}
	\le
	2^{(n+2)\lambda}4^{\lambda}k^{-\lambda}
	\sup_{\tau_{n+1}<\tau<t}
	\sum_{x\in V}
	(u(x,\tau)-\bar k_n)_+^\lambda\,\mu(x).
	\]
	The right-hand side is estimated using (\ref{firstit})
	with $\rho_1=3\cdot 2^{-n-4}$ and $\rho_2= 2^{-n-2}$,
	to obtain
	\[
	Y_{n+1}
	\le
	C
	2^{n(\lambda+\frac{\sigma}{\delta})}
	t^{-\frac \lambda\delta}
	k^{-\lambda}
	\Lambda_p(Y_n)^{-\frac \lambda\delta}
	Y_n.
	\]
	Using assumption (\ref{assumponLamb}) this yields
\begin{equation}\label{beforeCh}Y_{n+1}
	\le
	C
2^{n(\lambda+\frac{\sigma}{\delta})}
t^{-\frac \lambda\delta}
k^{-\lambda}
	\Lambda_p(Y_0)^{-\frac \lambda\delta}
	Y_0^{-\frac pN\frac \lambda\delta}
	Y_n^{1+\frac pN\frac \lambda\delta}.\end{equation}
Recall that \[
Y_0
=
\sup_{t/4<\tau<t}
\mu_V
\bigl(
\{x\in V:
u(x,\tau)>k/2\}
\bigr),
\]
and let
\[
E_r(t)
=
\sup_{0<\tau<t}
\sum_{x\in V}
u(x,\tau)^r
\mu(x).
\]
Clearly,
$
Y_0
\le
2^r k^{-r}E_r.
$
Let $a>0$. It follows from the assumptions in (\ref{assumponLamb}) that, for all $s\ge1$,
\[
\Lambda_p(sa)^{-1}
\le
s^\nu
\Lambda_p(a)^{-1}.
\]
Therefore, 
\[
\Lambda_p(Y_0)^{-\frac\lambda{\delta}}
\le
2^{r\frac{\nu\lambda}{\delta}}
\Lambda_p(k^{-r}E_r(t))^{-\frac\lambda{\delta}}.
\]	
Hence, we obtain from (\ref{beforeCh})		
\begin{equation}\label{iteration}Y_{n+1}\leq\frac{A^{n}Y_{n}^{1+\frac pN\frac \lambda\delta}}{\Theta }\end{equation}
where $$A=2^{\lambda+\frac{\sigma}{\delta}}\geq 1\quad\textnormal{and}\quad\Theta =ct^{\frac \lambda\delta}
k^{\lambda}
\Lambda_p(k^{-r}E_r(t))^{\frac \lambda\delta}
Y_0^{\frac pN\frac \lambda\delta}.$$
Now let us apply Lemma 6.1 from \cite{grigor2024finite} with $\omega =\frac pN\frac \lambda\delta$: if 
\begin{equation}
	\Theta \geq A^{\frac {N\delta}{p\lambda}}Y_{0}^{\frac pN\frac \lambda\delta },  \label{Thetayyy}
\end{equation}%
then, for all $n\geq 0,$ 
$
Y_{n}\leq A^{-n\frac {N\delta}{p\lambda}}Y_{0}.
$
The condition (\ref{Thetayyy}) is equivalent%
\begin{equation*}
ct^{\frac \lambda\delta}
k^{\lambda}
\Lambda_p(k^{-r}E_r(t))^{\frac \lambda\delta}
Y_0^{\frac pN\frac \lambda\delta}\geq A^{\frac {N\delta}{p\lambda}}Y_{0}^{\frac pN\frac \lambda\delta },
\end{equation*}%
that is,%
\begin{equation*}
	k\geq C
t^{-\frac1{\delta}}
\Lambda_p(k^{-r}E_r(t))^{-\frac1{\delta}}
.
\end{equation*}%
Let us choose $k$ to have equality here.
Recall that
$
\psi_r(s)
=
s^{\frac\delta r}
\Lambda_p(s^{-1}),
$
and let
$
\varphi_r
=
\psi_r^{-1}.
$
Then, by our choice of $k$, and the monotonicity properties of $\varphi_r$,
\[
k
=
E_r(t)^{1/r}
\left[
\varphi_r
\left(
C^{\delta}
t^{-1}
E_r(t)^{-\frac{\delta}{r}}
\right)
\right]^{1/r}\leq C'
E_r(t)^{1/r}
\left[
\varphi_r
\left(
t^{-1}
E_r(t)^{-\frac{\delta}{ r}}
\right)
\right]^{1/r}.
\]
For this $k $ we
obtain $Y_{n}\rightarrow 0$ as $n\rightarrow \infty $, which implies (\ref{MV}) by \cref{thm:mass-conservation}.

Now let us consider the case when $q<1$.
Let us set $w=u^{q}$ and $m=\frac{1}{q}>1$ so that $\partial_{t}w^{m}=\Delta_pw$. Testing with $\zeta(x)\eta(t)
(w(x,t)-k)_+^{\,\lambda-1}\mu(x)$, where $\lambda>1$, we obtain similarly to Lemma \ref{lemcacc}, for any $k\geq0$,	
\begin{align}\nonumber
	\left[\sum_{x\in V}
	H_{k}(w(x, t))\eta(t)\mu(x)\right]_{t_1}^{t_2}&+	\int_{t_1}^{t_2}
	\sum_{x,y\in V}
	\left|
	\nabla_{xy}
	\Bigl(
	(w(x,t)-k)_+^{(\lambda+p-2)/p}
	\Bigr)
	\right|^p\eta(t)
	\omega_{xy}\,dt\\&\nonumber\quad\leq 
	C\int_{t_1}^{t_2}
	\sum_{x\in V}
	H_{k}(w(x, t))\mu(x)\eta^{\prime}(t)\,dt,
\end{align}
where
$$
H_k(s) = 
\begin{cases}
	m \int_{k}^{s} r^{m-1}(r-k)^{\lambda-1} \, dr, & \text{for } s > k, \\
	0, & \text{for } s \leq k.
\end{cases}
$$ 
Set $\widetilde{\lambda}=m+\lambda-1$ and  $\widetilde{\delta}=p-1-m=\frac{\delta}{q}>0$ and $\widetilde{\sigma}=\widetilde{\lambda}+\widetilde{\delta}$ so that $\lambda+p-2=\widetilde{\sigma}$. Then we have $$H_{k}(w)\geq \frac{m}{\widetilde{\lambda}}(w-k)_{+}^{\widetilde{\lambda}}.$$
Now let $k_i$ and $t_i$ be defined as in (\ref{ki}) and (\ref{ti}) and let $
v_i(x,\tau)=(w(x,\tau)-k_i)_+.
$ 
If $w>k_i$, then $w-k_{i+1}>w-k_i$ and 
$$w={k_{i}}+(w-{k_{i}})\leq \frac{k_{i}}{k_{i}-k_{i+1}}v_{i+1}\leq \frac{2^{i+1}}{\rho_2-\rho_1}v_{i+1}.$$ Therefore, $$H_{{k_{i}}}(w)\leq mw^{m-1}v_{i}^{\lambda}\leq C\frac{2^{i(m-1)}}{(\rho_2-\rho_1)^{m-1}}v_{i+1}^{\widetilde{\lambda}}.$$ 
Applying the Faber--Krahn inequality \eqref{uFK} to
$v_{i+1}^{\widetilde\alpha}$, where
$\widetilde{\alpha}=\widetilde{\sigma}/p$, and then using H\"older's and
Young's inequalities exactly as in \eqref{appFK}, we obtain, with the same
time cutoff $\eta$,
\begin{align*}
&\sup_{t_i<\tau<t}
\sum_{x\in V}v_i(x,\tau)^{\widetilde{\lambda}}\,\mu(x)+	\int_{t_i}^t\sum_{x,y\in V}|\nabla_{xy}v_i(x,\tau)^{\widetilde{\alpha}}|^p\omega_{x,y}\,d\tau\\&\leq\frac{C2^{im}}{t(\rho_2-\rho_1)^m}\varepsilon^{\frac {\widetilde{\sigma}}{\widetilde{\lambda}}}	\int_{t_{i+1}}^t\sum_{x,y\in V}|\nabla_{xy}v_{i+1}(x,\tau)^{\widetilde{\alpha}}|^p\omega_{x,y}\,d\tau
+\frac{C2^{im}}{t(\rho_2-\rho_1)^m}
\varepsilon^{-\frac {\widetilde{\sigma}}{\widetilde{\delta}}}
t\Lambda_p(M_{i+1})^{-\frac {\widetilde{\lambda}}{\widetilde{\delta}}}
M_{i+1},
\end{align*}
where
$m_i(\tau)=\mu_V(\{x\in V:w(x,\tau)>k_i\})$ and
$M_i=\sup_{t_i<\tau<t}m_i(\tau)$.

Since for $Y_0
=
\sup_{t/4<\tau<t}
\mu_V
\bigl(
\{x\in V:
w(x,\tau)>k/2\}
\bigr)
$, $$Y_0\leq 2^{mr}k^{-mr}\sup_{0<\tau<t}
\sum_{x\in V}
u(x,\tau)^r
\mu(x)=2^{mr}k^{-mr}E_r(t),$$ 
we obtain, arguing as in (\ref{iteration}) for $Y_n
=
\sup_{\tau_n<\tau<t}
\mu_V
\bigl(
\{x\in V:\ w(x,\tau)>k_n\}
\bigr)$, $$Y_{n+1}\leq\frac{A^{n}Y_{n}^{1+\frac pN\frac {\widetilde{\lambda}}{\widetilde{\delta}}}}{\Theta }$$
where $$
A=2^{\widetilde{\lambda}+\frac{m\widetilde{\sigma}}{\widetilde{\delta}}}\geq 1\quad\textnormal{and}\quad\Theta =ct^{\frac{\widetilde{\lambda}}{\widetilde{\delta}}}
k^{\widetilde{\lambda}}
\Lambda_p(k^{-mr}E_r(t))^{\frac{\widetilde{\lambda}}{\widetilde{\delta}}}
Y_0^{\frac pN\frac{\widetilde{\lambda}}{\widetilde{\delta}}}.
$$
Hence, if 
\begin{equation}
	\Theta \geq A^{\frac {N\widetilde{\delta}}{p\widetilde{\lambda}}}Y_{0}^{\frac pN\frac{\widetilde{\lambda}}{\widetilde{\delta}}},  \label{Thetayyy2}
\end{equation}%
then, for all $n\geq 0,$ 
$
	Y_{n}\leq A^{-n\frac {N\widetilde{\delta}}{p\widetilde{\lambda}}}Y_{0}.
$
The condition (\ref{Thetayyy2}) is equivalent%
\begin{equation*}
	ct^{\frac{\widetilde{\lambda}}{\widetilde{\delta}}}
	k^{\widetilde{\lambda}}
	\Lambda_p(k^{-mr}E_r(t))^{\frac{\widetilde{\lambda}}{\widetilde{\delta}}}
	Y_0^{\frac pN\frac{\widetilde{\lambda}}{\widetilde{\delta}}}\geq A^{\frac {N\widetilde{\delta}}{p\widetilde{\lambda}}}Y_{0}^{\frac pN\frac{\widetilde{\lambda}}{\widetilde{\delta}} },
\end{equation*}%
that is,%
\begin{equation*}
	k\geq C
	t^{-\frac1{\widetilde{\delta}}}
	\Lambda_p(k^{-mr}E_r(t))^{-\frac1{\widetilde{\delta}}}
	.
\end{equation*}%
Let us choose $k$ to have equality here.
Consider the function
$
\Psi(s)
=
s^{\frac{\widetilde{\delta}}{r}}
\Lambda_p(s^{-m})=\psi_r(s^m),
$
and let
$
\Phi
=
\Psi^{-1}.
$
Then, by our choice of $k$,
\[
k
=
E_r(t)^{1/(mr)}
\Phi
\left(
C^{-\tilde{\delta}}
t^{-1}
E_r(t)^{-\frac\delta r}
\right)^{1/r}\leq C'
E_r(t)^{1/(mr)}
\Phi
\left(
t^{-1}
E_r(t)^{-\widetilde{\delta}/(mr)}
\right)^{1/r}.
\]
For this $k $ we
obtain $Y_{n}\rightarrow 0$ as $n\rightarrow \infty $, so that $$w(x, t)\leq  C'
E_r(t)^{1/(mr)}
\Phi
\left(
t^{-1}
E_r(t)^{-\frac\delta r}
\right)^{1/r},$$ that is, $$u(x, t)\leq C 'E_r(t)^{1/r}
\Phi
\left(
t^{-1}
E_r(t)^{-\frac\delta r}
\right)^{1/(qr)}.
$$ 
Finally, the identity $\Phi(s)^{m}=\varphi_r(s)$ implies that (\ref{MV}) also holds in this case.\end{proof}
    
    \subsection{Propagation of Finite Support Initial Data}
    
    \begin{lemma}\label{lemdecreasing}
    Assume that $\delta\geq0$ and set $\theta=\frac{N\delta}
    {(N\delta+p)(p-1)}$.
    Then the function
    $$
    \tau
    \mapsto
    \tau^\theta
    \,
    \varphi_1
    \!\left(
    \tau^{-1}b
    \right)^{\frac{\delta}{p-1}}
    $$
    is nondecreasing on $(0, \infty)$ for every $b>0$.
    \end{lemma}
    
    \begin{proof}
    Our claim is equivalent to showing that
    $
    r
    \mapsto
    r^{-\theta(p-1)}
    \varphi_1(r)^{\delta}
    $
    is nonincreasing.
    Let
    $
    s=\varphi_1(r).
    $
    Since
    $
    r=\psi_1(s),
    $
    we obtain
    \[
    r^{-\theta(p-1)}
    \varphi_1(r)^{\delta}
    =
    s^{(1-\theta(p-1))\delta}
    \Lambda_p(s^{-1})^{-\theta(p-1)}=\left(s^{-p/N}\Lambda_p(s^{-1})\right)^{-\theta(p-1)}.
    \]
    Using the assumption that
    $v^{-p/N}
    \Lambda_p(v^{-1})
    $
    is nondecreasing, the claim follows.
    \end{proof}

    The next result contains Theorem \ref{thm:propagation-finite-supportint} from Introduction.
    
    \begin{theorem}\label{thm:propagation-finite-support}
    	Assume graph $G$ satisfies the Faber--Krahn inequality and that
    	$q(p-1)>1$.	
    	Assume that $u_0\in \ell^1(V)$ be nonnegative and have its support contained in a ball $B_{R_{0}}$. Let $u\in L^\infty([0,T);\ell^1(V))$
    	be a solution of (\ref{eq:infinite-parabolic}).
    	Then for every $0<\varepsilon<1$ there
    		exists a constant $c>0$ such that
    		\begin{equation}\label{lowerforl1}\|u(t)\|_{\ell^1(B_R)}
    		\ge
    		(1-\varepsilon)
    		\|u_0\|_{\ell^1(V)}\end{equation}
    		whenever
    		\[
    		R
    		\ge
    		\max
    		\left\{
    		\frac{c}{\varepsilon}
    		t^{1/p}
    		\|u_0\|_{\ell^1(V)}^{\frac{\delta}{p}}
    		\varphi_1
    		\!\left(
    		t^{-1}
    		\|u_0\|_{\ell^1(V)}^{-\delta}
    		\right)^{\frac{\delta}{p}},
    		2R_0
    		\right\}.
    		\]
    \end{theorem}
    
    \begin{proof}
    Fix a vertex
    $
    x_0\in V.
    $ and let $
    B_R=B_R(x_0) 
    $ and
    $|x|=d(x,x_0)$.
    Let $R\geq 2 R_{0}$ so that 
    $\operatorname{supp}u_0\subset B_{R_0}\subset B_{R/2}$.
    Define $\varphi(x)
    =
    1-\zeta_{R/2,R}(x)$ and choose $
    \rho>4R$. 
    Multiplying both sides of the equation in (\ref{eq:infinite-parabolic}) by
    $
    \varphi(x)\zeta_{\rho,2\rho}(x)\mu(x)
    $, summing over $x$ and then integrating in $t$ yields
    \begin{align*}
    &\sum_{x\in V}
    \varphi(x)
    \zeta_{\rho,2\rho}(x)
    u(x,\tau)\mu(x)\\&=-\frac12
    \int_0^\tau
    \sum_{x,y\in V}
    |\nabla_{xy}(u^q(\cdot,t))|^{p-2}
    \nabla_{xy}(u^q(\cdot,t))
    \nabla_{xy}\!\left(
    \varphi(x)\zeta_{\rho,2\rho}(x)
    \right)
    \omega_{xy}\,dt.
    \end{align*}
    Notice that
    $
    \left|
    \nabla_{xy}\!\left(
    \varphi(x)\zeta_{\rho,2\rho}(x)
    \right)
    \right|
    \le
    |\nabla_{xy}\varphi(x)|
    \zeta_{\rho,2\rho}(y)
    +
    \varphi(x)
    |\nabla_{xy}\zeta_{\rho,2\rho}(x)|
    $.
    Hence,
    \[
    \left|
    \nabla_{xy}\!\left(
    \varphi(x)\zeta_{\rho,2\rho}(x)
    \right)
    \right|
    \le
    \frac{2}{R}
    +
    \frac{1}{\rho}.
    \]
    Thus,
    \begin{align*}
    \sum_{x\in V}
    \varphi(x)
    \zeta_{\rho,2\rho}(x)
    u(x,\tau)\mu(x)
    &\le
    \left[
    \frac{1}{R}
    +
    \frac{1}{2\rho}
    \right]
    \int_0^\tau
    \sum_{x,y\in V}
    |\nabla_{xy}(u^q(\cdot,t))|^{p-1}
    \omega_{xy}\,dt
    \\ &\le
    \frac{2}{R}
    \int_0^\tau
    \sum_{x,y\in V}
    |\nabla_{xy}(u^q(\cdot,t))|^{p-1}
    \omega_{xy}\,dt.
    \end{align*}
    Letting $\rho\to+\infty$, we obtain
    \[
    \sum_{x\notin B_{R}}
    u(x,\tau)\mu(x)
    \le
    \frac{2}{R}
    \int_0^\tau
    \sum_{x,y\in V}
    |\nabla_{xy}(u^q(\cdot,t))|^{p-1}
    \omega_{xy}\,dt.
    \]
    Let $\sigma>0$ be a constant to be chosen and define
    $
    \vartheta
    =
    2-\frac{1}{\delta+1}$.
    By Hölder's inequality,
    \begin{align*}
    \int_0^\tau
    \sum_{x,y\in V}
    |\nabla_{xy}u^q(x,t)|^{p-1}
    \omega_{xy}\,dt
    &\le
    \left(
    \int_0^\tau
    \sum_{x,y\in V}
    t^{-\sigma(p-1)}
    |u^q(x,t)+u^q(y,t)|^{(2-\vartheta)(p-1)}
    \omega_{xy}\,dt
    \right)^{1/p}
    \\&\times\left(
    \int_0^\tau
    \sum_{x,y\in V}
    t^\sigma
    |\nabla_{xy}u^q(x,t)|^p
    |u^q(x,t)+u^q(y,t)|^{\vartheta-2}
    \omega_{xy}\,dt
    \right)^{(p-1)/p}\\&=:I_2^{1/p}\times I_{3}^{(p-1)/p}.
    \end{align*}
    Choose $\sigma$ so that
    $
    \sigma(p-1)<1.
    $
    Then there exists a constant $C_1$ such that
    \[
    I_2
    \le
    2^{1/q+1}
    \int_0^\tau
    t^{-\sigma(p-1)}
    \|u(t)\|_{\ell^1(V)}\,dt
    \le
    C_1
    \|u_0\|_{\ell^1(V)}
    \tau^{1-\sigma(p-1)}.
    \]
    
    Multiply both sides of the equation in (\ref{eq:infinite-parabolic}) by
    $
    t^\sigma u^{q(\vartheta-1)}\mu(x),
    $
    integrate over $[0,\tau]$, and sum over $x$, we can obtain
    \begin{align*}
    \frac{1}{q(\vartheta-1)+1}
    &\sum_{x\in V}
    \mu(x)
    \left(
    \tau^\sigma u(x,\tau)^{q(\vartheta-1)+1}
    -
    \sigma
    \int_0^\tau
    t^{\sigma-1}
    u(x,t)^{q(\vartheta-1)+1}\,dt
    \right)
    \\&=
    -\frac12
    \int_0^\tau
    \sum_{x,y\in V}
    t^\sigma
    |\nabla_{xy}u^q(x,t)|^{p-1}
    \left|
    \nabla_{xy}\bigl(u(x,t)^{q(\vartheta-1)}\bigr)
    \right|
    \omega_{xy}\,dt.
    \end{align*}
    Rigorously, this test is justified by applying the spatial cutoff $\zeta_{L,2L}$ and integrating in time over $[\eta, \tau]$ with $\eta>0$. For $q(\vartheta-1)<1$, we avoid potential singularities by substituting $u^{q(\vartheta-1)}$ with
    $(u+h)^{q(\vartheta-1)}-h^{q(\vartheta-1)}$. The conclusion follows by taking the limits $L\to\infty$, $\eta\to0^+$, and $h\to0^+$.
    
    Hence, by Theorem \ref{thm:mass-conservation} and Theorem \ref{thm:upper-bound},
    \begin{align*}
    I_3
    &\le
    C\int_0^\tau
    \sum_{x,y\in V}
    t^\sigma
    |\nabla_{xy}u^q(x,t)|^{p-1}
    \left|
    \nabla_{xy}\bigl(u(x,t)^{q(\vartheta-1)}\bigr)
    \right|
    \omega_{xy}\,dt
    \\&\le
    C'
    \int_0^\tau
    \sum_{x\in V}
    t^{\sigma-1}
    u(x,t)^{q(\vartheta-1)+1}
    \mu(x)\,dt
    \\&\le
    C'\|u_0\|_{\ell^1(V)}
    \int_0^\tau
    t^{\sigma-1}
    \|u(t)\|_{\ell^\infty(V)}^{q(\vartheta-1)}
    \,dt.
    \\&\le
    C'
    \|u_0\|_{\ell^1(V)}^{q(\vartheta-1)+1}
    \int_0^\tau
    t^{\sigma-1}
    \varphi_1
    \left(
    t^{-1}
    \|u_0\|_{\ell^1(V)}^{-\delta}
    \right)^{q(\vartheta-1)}
    \,dt.
    \end{align*}
    Further, let us require that
    $\theta<\sigma<\frac1{p-1},$
    where $\theta$ is the constant defined in Lemma \ref{lemdecreasing}. Then we obtain by Lemma \ref{lemdecreasing},
    \[
    \int_0^\tau
    t^{\sigma-\theta-1}
    t^\theta
    \varphi_1
    \left(
    t^{-1}
    \|u_0\|_{\ell^1(V)}^{-\delta}
    \right)^{q(\vartheta-1)}
    \,dt
    \le
    \tau^\theta
    \varphi_1
    \left(
    \tau^{-1}
    \|u_0\|_{\ell^1(V)}^{-\delta}
    \right)^{q(\vartheta-1)}
    (\sigma-\theta)^{-1}
    \tau^{\sigma-\theta}.
    \]
    Combining the estimates for $I_2$ and $I_3$, we obtain
    \[
    \sum_{x\notin B_{R}}
    u(x,\tau)\mu(x)
    \le
    C\frac{1}{R}
    \tau^{1/p}
    \|u_0\|_{\ell^1(V)}^{1+\frac{\delta}{p}}
    \varphi_1
    \left(
    \tau^{-1}
    \|u_0\|_{\ell^1(V)}^{-\delta}
    \right)^{\frac{\delta}{p}}.
    \]
    Therefore,
    \[
    \|u(\tau)\|_{\ell^1(B_{R})}
    \ge
    (1-\varepsilon)\|u_0\|_{\ell^1(V)}
    \]
    provided
    \[
    R
    \ge
    \frac{c}{\varepsilon}
    \tau^{1/p}
    \|u_0\|_{\ell^1(V)}^{\frac{\delta}{p}}
    \varphi_1
    \left(
    \tau^{-1}
    \|u_0\|_{\ell^1(V)}^{-\delta}
    \right)^{\frac{\delta}{p}},
    \]
    which completes the proof.
    \end{proof}
    
	\section{Finite Time Extinction}\label{sec:conservation-extinction}
	
	We define the following edge seminorm for a function $f:V\to \mathbb{R}$ as follows:
	$$
	\Vert f\Vert_{\ell^p(E)}:=\left(\frac{1}{2}\sum_{x,y\in V}|f(x)-f(y)|^p\omega_{xy}\right)^{\frac{1}{p}}.
	$$
	The following Sobolev inequality is stated in \cite[Theorem~3.3]{Hua1} for $d\geq2$.
    Since its proof only uses $d>1$ and $1\leq p<d$, the same argument yields the following slightly general result.
	\begin{lemma}\label{thm:sobolev-isoperimetric}
		Assume that the graph $G$ satisfies the $d$-isoperimetric inequality for some $d>1$. Then there exists a constant $C=C(d,C_d)$ such that for all $p\in [1,d)$ and all functions $f$ with finite support on $V$,
		\begin{equation}
		    \Vert f\Vert_{\ell^{p^*}(V)}\leq C\frac{p}{d-p}\Vert f\Vert_{\ell^p(E)},
		\end{equation}
		where $p^*:=\frac{dp}{d-p}>p$ and $C_d$ is the constant mentioned in the definition of $d$-isoperimetric inequality.
	\end{lemma}

	\begin{proof}[Proof of \cref{thm:finite-time-extinction-infinite}]
    
        \textbf{Case 1.} $q>1$.
		Let $m=\frac{d(1-q(p-1))}{p}$, then we have $m>1$. By assumption, $u \in L^{\infty}([0,+\infty);\ell^r(V))$ is a solution of \eqref{eq:infinite-parabolic} with $u_0\in\ell^{m}(V)$. Then for all $t\geq 0$ and $k>0$, the function $(u(\cdot,t)-k)_{+}$ has finite support.
		
		For a given vertex $x\in V$, the following identity holds for
        almost every $t$:
		$$
		\frac{d}{dt}(u(x,t)-k)^m_{+}=m(u(x,t)-k)^{m-1}_{+}\frac{d}{dt}(u(x,t)-k)_{+}=m(u(x,t)-k)^{m-1}_{+}\frac{du}{dt}(x,t).
		$$
		Let $g(x,t):=(u-k)_{+}$ and $m'=\frac{m+q-1}{q}$. Since $|\nabla_{xy}(u^q(\cdot,t))|\geq |\nabla_{xy}(g^q(x,t))|$ when $q\geq 1$, we have
		\begin{align*}
			\frac{d}{dt} \sum_{x\in V}(u(x,t)-k)^m_{+}\mu(x) 
			&=m \sum_{x\in V}(u(x,t)-k)^{m-1}_{+}\sum_{y \sim x}\left(\nabla_{xy}(u^q(\cdot,t))\right)^{p-1}\omega_{xy}\\
			&=-\frac{m}{2}\sum_{x,y\in V}|\nabla_{xy}(u^q(\cdot,t))|^{p-1}\left|\nabla_{xy}\left((u(x,t)-k)^{m-1}_{+}\right)\right|\omega_{xy}\\
			&\leq -\frac{m}{2}\sum_{x,y\in V}|\nabla_{xy}(g^q(x,t))|^{p-1}\left|\nabla_{xy}\left(g^{m-1}(x,t)\right)\right|\omega_{xy}\\
			&\leq -\frac{m c(p,q,m)}{2}\sum_{x,y\in V} \left|\nabla_{xy}\left(g^{qs}(x,t)\right)\right|^p\omega_{xy}\\
			&= -mc(p,q,m)\ \Vert g^{qs}(\cdot,t)\Vert_{\ell^p(E)}^p.
		\end{align*}
		The last inequality follows from
        \cref{lem:power-difference} with $\varepsilon=m-1$. Here $s:=\frac{m'+p-2}{p}$ and $c(p,q,m)$ is a constant depending only on $p$, $q$, and $m$. Hence by \cref{thm:sobolev-isoperimetric}, there exists a constant $C'$ such that
		$$
		\begin{aligned}
			\frac{d}{dt} \sum_{x\in V}(u(x,t)-k)^{m}_{+}\mu(x)&\leq -C'\left(\sum_{x\in V}g^{qsp^*}(x,t)\mu(x)\right)^{\frac{p}{p^*}}\\
			&\leq -C'\left(\sum_{x\in V}(u(x,t)-k)^{qs p^*}_{+}\mu(x)\right)^{\frac{p}{p^*}},
		\end{aligned}
		$$
		where $p^*=\frac{dp}{d-p}$. Since $m=\frac{d(1-q(p-1))}{p}$ and $m'=\frac{m+q-1}{q}$, we have
        $$
        qs p^*=\frac{q\bigl(m'+p-2\bigr)}{p} \frac{dp}{d-p}
        =
        \frac{d\bigl(m+q(p-1)-1\bigr)}{d-p}
        =m.
        $$
		Let $E_{m,k}(t):=\sum_{x\in V}(u(x,t)-k)^m_{+}\mu(x)$. We then obtain that
		$$
		E_{m,k}'(t)+C'E_{m,k}(t)^{\frac{p}{p^*}}\leq 0.
		$$ 
		Integrating over $t$, we have
		$$
		E_{m,k}(t)\leq\left(E_{m,k}(0)^{p/d}-\frac{C'p}{d}t\right)^{d/p}\quad \text{when}\quad E_{m,k}(t)>0.
		$$
		Let $t_{0,k}:=\frac{d}{C'p}E_{m,k}(0)^{p/d}$, since $E_{m,k}(t)$ is nonnegative and nonincreasing, we conclude that $E_{m,k}(t)=0$ for every $t\geq t_{0,k}$, which implies that $u(x,t)\leq k$ for every $t\geq t_{0,k}$. Let 
		$$
		t_0 :=\lim\limits_{k\to 0}t_{0,k}=\frac{d}{C'p}{\Vert u_0\Vert_{\ell^m(V)}^{mp/d}}.
		$$
		Then, we can see that $u(x,t)\leq 0$ for every $t\geq t_0$. Since $-u$ is also a solution of \eqref{eq:infinite-parabolic}, we have $-u(x,t)\leq 0$. Thus, $u(x,t)\equiv 0$ for every $t\geq t_{0}$, which completes the proof.

        \textbf{Case 2.} $0 < q \leq 1$ and $u$ is an exhaustion solution. Fix $x_0\in V$. Without loss of generality, we may assume that $u(x,t)=\lim_{n\to\infty} u_{n}(x,t)$ for all $(x,t)\in V\times[0,+\infty)$, where each $u_n:V\times[0,\infty)\to \mathbb{R}$ is a solution of the mixed problem \eqref{eq:finite-ball-problem}.

        Let $m=\frac{d(1-q(p-1))}{p}$. Define
        $$
        E_{m,n}(t):=\sum_{x\in V}|u_n(x,t)|^m\mu(x).
        $$
        Since $u_n$ has finite support, after differentiation we obtain that
        $$
        \frac1m E_{m,n}'(t)
        =
        \sum_{x\in V}u_n(x,t)^{m-1}\frac{\partial u_n}{\partial t}(x,t)\mu(x).
        $$
        Using the equations in \eqref{eq:finite-ball-problem} and summing over $x\in V$ yields
        $$
        \frac1m E_{m,n}'(t)
        =
        -\frac12\sum_{x,y\in V}
        \left(\nabla_{xy}(u_n^q)(x,t)\right)^{p-1}\,
        \nabla_{xy}(u_n^{m-1})(x,t)\,\omega_{xy}.
        $$
        Applying the same argument as in Case 1, we obtain
        $$
        E_{m,n}'(t)
        \le
        -\frac{mc(p,q,m)}{2}
        \sum_{x,y\in V}\left|\nabla_{xy}(u_n^{qs})(x,t)\right|^p\omega_{xy}=-mc(p,q,m)\|u_n^{qs}(\cdot,t)\|^p_{\ell^p(E)},
        $$
        where $s=\frac{m'+p-2}{p}$.
        By \cref{thm:sobolev-isoperimetric}, there exists a constant $C'$ such that
        $$
        \begin{aligned}
            E_{m,n}'(t)&\leq -mc(p,q,m)\|u_n^{qs}(\cdot,t)\|^p_{\ell^p(E)}\\
            &\leq -
            C'
            \left(\sum_{x\in V}|u_n(x,t)|^{qs p^*}\mu(x)\right)^{p/p^*}=-C'E_{m,n}(t)^{p/p^*}.
        \end{aligned}
        $$
        Again, by a similar argument as in Case 1, we can obtain that $E_{m,n}(t)=0$ for all $t\geq t_{0,n}$, where
        $$
        t_{0,n}=\frac{d}{C'p}E_{m,n}(0)^{p/d}.
        $$
        Let $t_0:=\frac{d}{C'p}{\Vert u_0\Vert_{\ell^m(V)}^{mp/d}}$. Since for all $n\in\mathbb{N}$, $E_{m,n}(0)\leq \|u_0\|_{\ell^m(V)}^{m}$, it follows that each $u_n$ vanishes identically for all $t\geq t_0$. Therefore, the exhaustion solution $u$ obtained as pointwise limit of $\{u_n\}_n$ must also vanish identically for all $t\geq t_0$. This completes the proof.
	\end{proof}

	\section{A Classification of Solutions on Finite Graphs}\label{sec:finite-graph-classification}
    Now assume that $G=(V,E,\omega)$ is a finite graph and fix a proper subset $\Omega\subsetneq V$ such that the induced subgraph $G[\Omega]$ is connected. We consider the following equation:
    \begin{equation}\label{eq:finite-parabolic}
    \begin{cases}
        \dfrac{\partial u}{\partial t}(x,t) = \Delta_p u^q(x,t), & x\in\Omega,\ t>0,\\[6pt]
        u(x,0) = u_0(x), & x\in\Omega,\\[6pt]
        u(x,t) = 0, & x\in V\setminus\Omega,\ t\ge 0,
    \end{cases}
    \end{equation}
    where $u_0$ is a nonnegative function on $\Omega$ that is not identically zero. As before, the graph $G$ is always assumed to be connected, without multiple edges or self-loops, and we always take $p>1$ and $q>0$.
    
    \begin{definition}\label{def:finite-graph-solution}
        Let $G$ be a finite graph. A function $u: V\times[0,+\infty)\to\mathbb{R}$ is called a solution of \eqref{eq:finite-parabolic} if $u(x,\cdot)\in C^1([0,+\infty))$ for every $x\in V$ and $u$ satisfies \eqref{eq:finite-parabolic} pointwise.
    \end{definition}
    
    The case $q=1$ was studied in \cite{Lee}, where it was shown that the solution vanishes in finite time when $p<2$, while it remains strictly positive on $\Omega$ for $p\ge 2$. In this section, we show that the same dichotomy holds for all $q>0$. 
    \begin{theorem}\label{thm:finite-graph-dichotomy}
    Let $u$ be a solution of \eqref{eq:finite-parabolic}. 
    \begin{enumerate}[(i)]
        \item If $q(p-1)<1$, then there exists a finite time $t_0>0$ such that $u(x,t)=0$ for all $x\in V$ and all $t\ge t_0$.
        \item If $q(p-1)\ge 1$, then $u(x,t)>0$ for all $x\in\Omega$ and all $t>0$.
    \end{enumerate}
    \end{theorem}
    To establish this result, we introduce the following notation:
    $$
    \begin{aligned}
        \Omega_{T} & :=\Omega \times(0, T], \\
        V_{T} & :=V \times[0, T], \\
        \Gamma_{T} & :=V_{T}\setminus\Omega_{T}=(\Omega \times\{t=0\}) \cup ((V\setminus\Omega) \times[0, T]). 
    \end{aligned}
    $$
    
    \begin{lemma}[(Maximum Principle)]\label{lem:maximum-principle}
        Suppose that a function $u: V \times [0,T] \rightarrow \mathbb{R}$ satisfies $u(x,\cdot)\in C^1([0,T])$ for all $x\in V$. If 
        $$
        \frac{\partial u}{\partial t}-\Delta_{p} u^q \leq 0 \quad \text{in } \Omega_{T},
        $$
        then
        $$
        \max_{V_{T}} u=\max_{\Gamma_{T}} u.
        $$
        Note that this result also holds when $\Omega=V$.
    \end{lemma}
    
    \begin{proof}
        Since $u(x,\cdot)$ is continuous on $[0,T]$ for every $x\in V$, there exists $(x_{0}, t_{0}) \in V_{T}$ such that $u(x_{0}, t_{0})=\max_{V_{T}} u$. Without loss of generality, we assume that $(x_0,t_0)\notin \Gamma_T$. Since $u$ attains its maximum at $(x_{0}, t_{0})$, we have
        $$
        \Delta_{p} u^q(x_{0}, t_{0})=\sum_{y \in V}
        |u^q(y, t_{0})-u^q(x_{0}, t_{0})|^{p-2}[u^q(y, t_{0})-u^q(x_{0}, t_{0})] \frac{\omega_{x_0y}}{\mu(x_{0})} \leq 0.
        $$
        It is also clear that $\frac{\partial u}{\partial t}(x_0, t_0)\geq 0$. Thus, $\frac{\partial u}{\partial t}-\Delta_{p} u^q\geq 0$ at $(x_0,t_0)$. Combining this with the condition $\frac{\partial u}{\partial t}-\Delta_{p} u^q \leq 0$ yields $\frac{\partial u}{\partial t}-\Delta_{p} u^q = 0$ at $(x_0, t_0)$. Hence,
        $$
        0\leq\frac{\partial u}{\partial t}(x_0, t_0)=\Delta_{p} u^q(x_{0}, t_{0})\leq 0.
        $$
        It follows that $\Delta_{p} u^q(x_{0}, t_{0})=0$, which implies that $u(y,t_0)=u(x_{0}, t_{0})=\max_{V_{T}} u$ for all $y\sim x_0$.
        
        When $\Omega\subsetneq V$, due to the connectivity of $G$, there must exist $x_1\in V\setminus\Omega$ such that $u(x_1,t_0)=\max_{V_{T}} u$. Consequently, $\max_{V_{T}} u=\max_{\Gamma_{T}} u$.
        When $\Omega=V$, again due to the connectivity of $G$, we have $u(x,t_0)=\max_{V_{T}} u$ for all $x\in V$. Since
        $$
        \sum_{x\in V}\frac{\partial u}{\partial t}(x,t)\mu(x)\leq\sum_{x, y \in V}\left|\nabla_{xy}u^q\right|^{p-2}\nabla_{xy}u^q\,\omega_{xy}=0,
        $$
        the function $\sum_{x\in V}u(x,t)\mu(x)$ is non-increasing with respect to $t$, and thus
        $$
        \sum_{x\in V}u_0(x)\mu(x)\geq\sum_{x\in V}u(x,t_0)\mu(x)=\left(\sum_{x\in V}\mu(x)\right)\max_{V_{T}} u.
        $$
        Consequently, we have 
        $$
        \max_{V_{T}} u\leq\max_{x\in V}u_0(x)=\max_{\Gamma_{T}} u.
        $$
        This completes the proof.
    \end{proof}
    
    \begin{remark}\label{rem:finite-graph-alternative}
        From \cref{lem:maximum-principle}, one can also derive the minimum principle and, consequently, the nonnegativity of solutions. Due to the connectivity of the subgraph $G[\Omega]$, similar to \cref{rem:nonnegative-alternative}, we observe that solutions to \eqref{eq:finite-parabolic} are either identically positive on $\Omega$ or vanish identically after some finite time $t_0$.
    \end{remark}
    
    \begin{proof}[Proof of \cref{thm:finite-graph-dichotomy}]
        We first consider the case where $q(p-1)<1$. By choosing $d>p$ sufficiently large, we can ensure that $q(p-1) <\frac{d-p}{d}$. Since $V$ is finite and $\Omega\subsetneq V$, there exists a constant $C>0$ such that for any function $f: V \to \mathbb{R}$ with support $\operatorname{supp} f\subset \Omega$, we have
        $$
            \Vert f\Vert_{\ell^{p^*}(V)}\leq C\Vert f\Vert_{\ell^p(E)},
        $$
        where $p^*=\frac{dp}{d-p}$. The conclusion then follows from the exact same argument as in the proof of \cref{thm:finite-time-extinction-infinite}.
    
        Now we turn to the case where $q(p-1)\ge 1$. We proceed by contradiction. Suppose that there exist some vertex $x\in\Omega$ and time $t>0$ such that $u(x,t)=0$. Define 
        $$
        t_0:=\inf\{t>0 : \exists x\in \Omega \text{ such that } u(x,t)=0\}.
        $$
        Since $u_0\neq0$, it is evident that $t_0>0$. According to \cref{rem:finite-graph-alternative}, we must have $u(\cdot,t)\equiv 0$ for all $t\geq t_0$.
        Let $k>0$ be a constant to be determined later, define the function $h(t)=t^k\sum_{x\in V}u(x,t)\mu(x)$. Differentiating $h(t)$ with respect to $t$ yields
        $$
		\begin{aligned}
			h'(t)&=kt^{k-1}\sum_{x\in V}u(x,t)\mu(x)+t^k\sum_{x\in V}\frac{d }{dt}u(x,t)\mu(x)\\
			&=t^k\left(\frac{k}{t}\sum_{x\in \Omega}u(x,t)\mu(x)+\sum_{x\in \Omega}\sum_{y\sim x}\left(\nabla_{xy}(u^q(\cdot,t))\right)^{p-1}{\omega_{xy}}\right)\\
			&=t^k\left(\frac{k}{t}\sum_{x\in \Omega}u(x,t)\mu(x)-\sum_{x\in \Omega}\sum_{y\in V\backslash\Omega}u^{q(p-1)}(x,t){\omega_{xy}}\right).
		\end{aligned}
		$$
        Since $u(\cdot,t)\to 0$ as $t\to t_0^-$, there exists $\varepsilon\in(0,t_0)$ such that $0 \le u(x,t)<1$ for all $x\in\Omega$ and $t\in[t_0-\varepsilon,t_0)$. Thus, for any fixed $k > t_0$ and all $t\in[t_0-\varepsilon,t_0)$,
        $$
        \sum_{x\in \Omega}\sum_{y\in V\setminus\Omega}u^{q(p-1)}(x,t)\omega_{xy} \leq \sum_{x\in\Omega}u^{q(p-1)}(x,t)\mu(x) \leq \sum_{x\in\Omega}u(x,t)\mu(x)\leq \frac{k}{t}\sum_{x\in\Omega}u(x,t)\mu(x).
        $$
        Therefore, we have $h'(t)\geq 0$ on the interval $[t_0-\varepsilon,t_0)$, which implies that $h(t)$ is non-decreasing on this interval. Since $h(t)>0$ for all $0<t<t_0$, we must have
        $$
        h(t_0) = \lim_{t\to t_0^-} h(t) \ge h(t_0-\varepsilon) > 0.
        $$
        However, by definition we have $u(\cdot,t_0)\equiv 0$ and therefore $h(t_0)=0$, which is a contradiction. Hence, $u(x,t)$ remains strictly positive on $\Omega$ for all $t>0$.
    \end{proof}

    \appendix
    \section{Faber--Krahn Inequality}\label{sec:appendix}
    \phantomsection 
    \setcounter{section}{1}
    \setcounter{theorem}{0}
    
	In this appendix we establish the existence of $\Lambda_p$ functions for some specific classes of graphs.
	
	\begin{proposition}
    	\label{prop:Faber--Krahn-polynomial}
    		Let $G=(V,E,\omega)$ be an infinite connected graph. Suppose that $A_1:=\inf_{x\sim y}\frac{\omega_{xy}}{\mu(x)}>0$ and $\mu_0:=\inf_{x\in V}\mu(x)>0$. Then for each $p\geq 1$, $G$ satisfies the Faber--Krahn inequality  with the function $\Lambda_p(v)=\gamma_0 v^{-p}$, where $\gamma_0>0$ is a constant depending only on $p, A_1$, and $\mu_0$.
    	\end{proposition}
    	\begin{proof}
    	(1). We first consider the case $p=1$. Let $U$ be an arbitrary finite subset of $V$ and $f$ be an arbitrary function with support $\operatorname{supp}(f)\subset U$. Fix an arbitrary vertex $x_0\in V\setminus U$. For each $x\in U$, by connectivity, there exists a simple path 
        $$
        x=y_0\sim y_1\sim\cdots\sim y_k\sim y_{k+1}=x_0
        $$
        in $G$ such that all vertices in the path are distinct. Since $f(x_0)=0$, we have 
    	$$
    	|f(x)|\leq|f(x)-f(y_1)|+\sum_{i=1}^{k-1}{|f(y_i)-f(y_{i+1})|}+|f(y_k)-f(x_0)|\leq \sum_{\{u, v\}\in E}|f(v)-f(u)|,
    	$$
    	which implies that
    	\begin{equation}\label{eq:F-K-p-1}
    	    \sum_{x\in U}|f(x)|\mu(x)\leq\mu_V(U)\sum_{\{u, v\}\in E}|f(v)-f(u)|\leq \frac{\mu_V(U)}{A_1\mu_0} \sum_{u, v\in V}|f(v)-f(u)|\omega_{uv}.
    	\end{equation}
    	Thus, we can choose $\Lambda_1(v)=A_1\mu_0v^{-1}$.
    	
    	(2). Now consider the case $p>1$. Applying the inequality \eqref{eq:F-K-p-1} with $f$ replaced by $|f|^p$, we obtain
    	$$
    	\begin{aligned}
    		\sum_{x\in U}|f(x)|^p\mu(x) &\leq\frac{\mu_V(U)}{A_1\mu_0} \sum_{x, y\in V}\left||f(y)|^p-|f(x)|^p\right|\omega_{xy}\\
    		&\leq\frac{p\mu_V(U)}{A_1\mu_0} \sum_{x, y\in V}|f(y)-f(x)|\left(|f(y)|^{p-1}+|f(x)|^{p-1}\right)\omega_{xy}\\
    		&\leq\frac{p\mu_V(U)}{A_1\mu_0} \left(\sum_{x, y\in V}|f(y)-f(x)|^p\omega_{xy}\right)^{\frac{1}{p}} \times \\
    		&\quad\quad\quad\quad \left(\sum_{x, y\in V}\left(|f(y)|^{p-1}+|f(x)|^{p-1}\right)^{\frac{p}{p-1}}\omega_{xy}\right)^{\frac{p-1}{p}}\\
    		&\leq \frac{4p\mu_V(U)}{A_1\mu_0} \left(\sum_{x, y\in V}|f(y)-f(x)|^p\omega_{xy}\right)^{\frac{1}{p}}\left(\sum_{x\in U}|f(x)|^p\mu(x)\right)^{\frac{p-1}{p}},
    	\end{aligned}
    	$$
    	where the second inequality uses \cref{lem:power-lipschitz} and the third inequality applies H\"older's inequality. Dividing both sides by $\left(\sum_{x\in U}|f(x)|^p\mu(x)\right)^{\frac{p-1}{p}}$ yields
    	$$
    	\left(\sum_{x\in U}|f(x)|^p\mu(x)\right)^{\frac{1}{p}}\leq\frac{4p\mu_V(U)}{A_1\mu_0} \left(\sum_{x, y\in V}|f(y)-f(x)|^p\omega_{xy}\right)^{\frac{1}{p}}.
    	$$
    	Hence, we can set $\Lambda_p(v)=\gamma_0v^{-p}$ with $\gamma_0=\left(\frac{A_1\mu_0}{4p}\right)^p$.
    \end{proof}

    Let $G=(\Gamma, S)$ be the Cayley graph of a discrete group ${\Gamma}$ with a finite symmetric generating set $S$, where the weights $\omega$ satisfy $\omega_{xy}=1$ whenever $x\sim y$. Recall that 
    $$ 
    B_R = \{x \in V : d (x, e) \leq R\}
    $$
    is the ball centered at the identity element $e$ with radius $R>0$. 
    The following result is proved by Coulhon and Saloff-Coste in \cite{coulhon-isoperimetric}.
    \begin{lemma}\label{lem:Cayley-isoperimetric}
        Let $\mathcal{V}(R)$ be a continuous, non-negative, strictly increasing function on $[0, +\infty)$ such that $\mathcal{V}(R) \to \infty$ as $R \to \infty$. Assume that for all non-negative integers $R$,
        $$
        \mu_V(B_R) \ge \mathcal{V}(R).
        $$
        Then the Cayley graph $(\Gamma, S)$ satisfies the following isoperimetric inequality
        $$
        \mu_{E} (\partial\Omega) \geq \Phi (\mu_V (\Omega)),
        $$
        with function
        $$
        \Phi(s) = \frac{c_0 s}{\mathcal{V}^{-1}(2s)}, \quad s \ge \vert S\vert,
        $$
        where $c_0 = \frac{ \mathcal{V}^{-1}(2|S|)}{4|S| \left(\mathcal{V}^{-1}(2|S|) + 1 \right)}$.
    \end{lemma}

	\begin{proposition}\label{thm:Cayley-Faber--Krahn}
		Let $G=(\Gamma,S)$ be a Cayley graph. Suppose that there exist constants $c,N>0$ such that 
        $$
        \mu_V(B_r) \ge c r^N,\quad\forall r\geq 0.
        $$
        Then for each $p\geq1$, $G$ satisfies the Faber--Krahn inequality with the function 
        $$
        \Lambda_p(v)=\gamma_0v^{-\frac{p}{N}},
        $$ 
        where $\gamma_0$ is a sufficiently small constant.	
    \end{proposition}
	\begin{proof}
        In view of the proof of \cref{prop:Faber--Krahn-polynomial}, we only need to prove the case $p=1$.
        For any finite subset $U$ of the vertex set $V$, let
        \[
            \lambda_{1,1}(U) = \inf_{f \in \mathbb{R}^U \setminus \{0\}} \frac{\Vert f \Vert_{\ell^1(E)}}{\Vert f \Vert_{\ell^1(V)}}, \quad \text{and} \quad h_{\mu,\nu}(U) = \min_{\varnothing \neq \Omega \subset U} \frac{\mu_E(\partial\Omega)}{\mu_V(\Omega)},
        \]
        where $\partial\Omega = \{\{x, y\} \in E \mid x \in \Omega, \, y \in V \setminus \Omega\}$. By \cite[Proposition 1.1]{Hua2}, we have $\lambda_{1,1}(U) = h_{\mu,\nu}(U)$ for every finite subset $U \subset V$. Therefore, it suffices to prove that 
        \[
            \mu_E(\partial\Omega) \ge \gamma_0 \mu_V(\Omega)^{1-\frac{1}{N}}
        \]
        for any finite subset $\Omega \subset V$. Since $\mu_V(B_r) \ge c r^N$, this is a direct consequence of \cref{lem:Cayley-isoperimetric}.
    \end{proof}

    \textbf{Acknowledgements.} 
    The authors would like to express their gratitude to Professor Yuhua Sun for proposing the initial research problem and providing invaluable guidance throughout the preparation of this paper. The authors also thank Professor Alexander Grigor'yan for his valuable comments and suggestions on the manuscript.
    \bibliographystyle{alpha}
    \bibliography{main}

\end{document}